\documentclass[11pt,a4paper]{article}

\usepackage[utf8]{inputenc}

\usepackage{amsmath,amssymb,amsthm}
\usepackage{amsfonts}
\usepackage{amsmath}
\usepackage{mathrsfs}
\usepackage{array}
\usepackage{float}
\usepackage{pgfplots}
\usepackage{tikz}
\usetikzlibrary{arrows.meta, calc, patterns}

\usepackage{array, longtable, booktabs}
\usepackage{amssymb}
\usepackage{ragged2e}  

\newcolumntype{L}[1]{>{\RaggedRight}p{#1}}
 
\usepackage{mathrsfs}

\usepackage{enumitem}

\usepackage{pgfplots}
\usetikzlibrary{patterns, decorations.pathreplacing, calc}
\usepgfplotslibrary{fillbetween}
\pgfplotsset{compat=1.18}

\usepackage{amsmath}  
\usepackage{tabularx}  
\usepackage{booktabs}  
\usepackage{color}
\usepackage{ascmac}
\usepackage{dsfont}

\usepackage{cite}

\usepackage{comment}
\usepackage{todonotes}

\usepackage[title]{appendix}

\usepackage{geometry}
\theoremstyle{definition}
 \newtheorem{dfn}{Definition}[section]
 \newtheorem{remark}[dfn]{Remark}
  \newtheorem{rmk}[dfn]{Remark}

\theoremstyle{plain}
 \newtheorem{thm}[dfn]{Theorem}
 \newtheorem{prop}[dfn]{Proposition}
 \newtheorem{lem}[dfn]{Lemma}
 \newtheorem{cor}[dfn]{Corollary}

\usepackage{hyperref}

\newcommand{\bn}{{\mathbf n}}

\newcommand{\ba}{{\mathbf a}}
\newcommand{\bb}{{\mathbf b}}

\newcommand{\bA}{{\mathbf A}}

\newcommand{\bF}{{\mathbf F}}

\newcommand{\bU}{{\mathbf U}}

\newcommand{\dv}{{\rm div}\,}

\newcommand{\BN}{{\mathbb N}}

\newcommand{\CP}{{\mathcal P}}
\newcommand{\CQ}{{\mathcal Q}}

\newcommand{\fp}{{\mathfrak p}}

\newcommand{\be}{{\mathbf e}}

\newcommand{\pd}{\partial}

\usepackage{xcolor}

\newcommand{\pa}{\partial}
\newcommand{\wt}{\widetilde}

\newcommand{\am}{{\rm am}\,}
\newcommand{\dn}{{\rm dn}\,}

\numberwithin{equation}{section}

\usepackage{empheq}

\usepackage{accents}

\newcommand{\R}{{\mathbb R}}

\newcommand{\Rs}{\R_{\star}^2}

\usepackage{scalerel,stackengine}
\stackMath
\newcommand\reallywidehat[1]{%
\savestack{\tmpbox}{\stretchto{%
  \scaleto{%
    \scalerel*[\widthof{\ensuremath{#1}}]{\kern-.6pt\bigwedge\kern-.6pt}%
    {\rule[-\textheight/2]{1ex}{\textheight}}
  }{\textheight}%
}{0.5ex}}%
\stackon[1pt]{#1}{\tmpbox}%
}

\newcommand{\vertiii}[1]{{\left\vert\kern-0.25ex\left\vert\kern-0.25ex\left\vert #1 
    \right\vert\kern-0.25ex\right\vert\kern-0.25ex\right\vert}}

\begin{document}

\title{\bf  Local profiles of self-similar solutions of the planar stationary Navier--Stokes equations}


\author{Ming Li$^1$\thanks{
Email: mingli\_happy@163.com},
~Linyu Peng$^2$\thanks{
Email: l.peng@mech.keio.ac.jp},
~Ping Zhang$^{3,4}$\thanks{
Email: zp@amss.ac.cn},
~and Xin Zhang$^5$\thanks{
Email: xinzhang2020@tongji.edu.cn
}
\vspace{0.4cm}
\\
{\it 1. School of Mathematical Sciences,
Tongji University,}\\
{\it No.1239, Siping Road, Shanghai 200092, China}\\
{\it 2. Department of Mechanical Engineering, Faculty of Science and Technology, } \\
{\it Keio University, Yokohama 223-8522, Japan}\\
{\it 3. State Key Laboratory of Mathematical Sciences,}\\
{\it Academy of Mathematics $\&$ Systems Science, The Chinese Academy of
	Sciences,}\\
{\it  Beijing 100190, China}\\
    {\it 4. School of Mathematical Sciences,
University of Chinese Academy of Sciences,}\\
{\it Beijing 100049, China}\\
{\it 5. School of Mathematical Sciences,}\\
{\it Key Laboratory of Intelligent Computing and Applications (Ministry of Education),} \\
{\it Tongji University, 
No.1239, Siping Road, Shanghai 200092, China}
}


\maketitle
\begin{abstract}
In this paper, we revisit self-similar solutions of the two-dimensional stationary incompressible Navier--Stokes equations under scaling symmetries, also known as \emph{Jeffery--Hamel solutions}. We investigate the local patterns of smooth Jeffery--Hamel solutions in a conical subdomain $\Omega$ with vertex at the origin, \emph{without} imposing any boundary conditions on $\Omega$. For  radial Jeffery--Hamel solutions, we obtain \emph{all} the explicit local profiles in $\Omega$ with \emph{arbitrary} opening angles. In the non-radial case, we show that some Jeffery--Hamel solutions can be obtained via solving a Li\'enard equation, and we derive new explicit local profiles expressible in terms of Weierstrass elliptic functions. 

\end{abstract}

\tableofcontents

\section{Introduction}
\subsection{Model and history}
\label{sec:1.1}
Let $\R_{\star}^2:=\R^2\backslash \{(0,0)\}$. This paper concerns the stationary incompressible Navier--Stokes (SNS) equations in some cone $\Omega\subset \R_{\star}^2$ with its vertex at the origin, especially for the case with no external force: 
\begin{align}
\label{eq:NS_0}
    \begin{cases}
       \dv (\bU \otimes \bU) -\Delta \bU + \nabla p = 0,\\
      \dv  \bU = 0,
    \end{cases}
\end{align}
where the unknowns $\bU=\big(u(x,y),v(x,y)\big)$ and $p=p(x,y)$ denote the velocity vector and the pressure of the fluids at a point $(x,y)\in \Omega$, respectively. In \eqref{eq:NS_0}, the symbol $\ba\otimes \bb$ stands for the $2\times 2$ matrix $[a_jb_k]$ for any vectors $\ba,\bb \in \R^2$, and the $j$th component of the vector $\dv \bA$ is $\sum_{k=1}^2 \pd_k A_{jk}$ for any $2\times 2$ matrix $\bA=[A_{jk}]$. 

It is well known that the system \eqref{eq:NS_0} admits the following scaling invariance property, that is, 
\begin{equation}\label{eq:scasym}
    \big( \bU_{\lambda}(x,y),p_{\lambda}(x,y)\big)
    :=\big( \lambda \bU (\lambda x, \lambda y), \lambda^2 p(\lambda x, \lambda y) \big)
\end{equation}
also satisfies \eqref{eq:NS_0} for any $\lambda>0$,
whenever $(\bU, p)$ solves \eqref{eq:NS_0}. We are interested in \emph{self-similar solutions} of \eqref{eq:NS_0} with respect to \eqref{eq:scasym} satisfying the following condition 
\begin{equation}\label{eq:ss}
(\bU, p)=(\bU_{\lambda},p_{\lambda}) \quad   \text{for any}\,\,\,\lambda>0.
\end{equation}
By \eqref{eq:ss}, every self-similar solution of 
\eqref{eq:NS_0} in $\R_{\star}^2$ is fully determined by its trace on the circle
$$\mathbb{S}^1:= \big\{(x,y)\in \Rs: \sqrt{x^2+y^2}=1\big\}.$$
Thus, it is reasonable to use the polar coordinates  $(x,y)=(r\cos \theta, r\sin \theta)$ for $r:=\sqrt{x^2+y^2}$ and $\theta\in (-\pi,\pi]$.
Denoting
$\be_{r}:=(\cos \theta, \sin \theta)^{\top}$ and $\be_{\theta}:=(-\sin \theta, \cos \theta)^{\top}$,
the self-similar solution $(\bU, p)$ of \eqref{eq:NS_0} then admits the following form
\begin{equation}\label{eq:Up}
    \bU = \frac{f(\theta)}{r} \be_{r} +\frac{g(\theta)}{r} \be_{\theta}, \quad p=\frac{\fp (\theta)}{r^2}
\end{equation}
in view of \eqref{eq:ss}.  Here, the functions $f,g$ and $\fp$ are scalar about $\theta$ satisfying the ordinary differential equations (ODEs; see \cite[System (12)]{sverak2011landau}):
\begin{equation}\label{eq:pfg}
  \begin{aligned}
    \begin{cases}
       (\fp-2f)' = 0,\\
      -f''+gf'-f^2-g^2-2\fp = 0,\\
      g'=0.
    \end{cases}
\end{aligned} 
\end{equation}

The nontrivial\footnote{Namely, the velocity field $\bU=(u,v) \not\equiv 0.$}
smooth self-similar solutions of \eqref{eq:NS_0} are exactly characterized by the so-called \emph{Jeffery--Hamel solutions} due to the pioneering studies of exact general solutions of \eqref{eq:NS_0} by Jeffery \cite{jeffery1915two} and Hamel \cite{hamel1917spiralformige}.

\begin{dfn}[Jeffery--Hamel solutions in $\Rs$]
\label{def:JH}
We call that $(\bU,p)$ in the form of \eqref{eq:Up} a Jeffery--Hamel solution of \eqref{eq:NS_0} in $\Rs$ if either of the following conditions holds:
    \begin{itemize}
    \item $f(\theta)$ and $g(\theta)$ are constants. In other words, for any constants $C_1,C_2\in \R$,
\begin{equation}\label{sol:CC}
    \bU =\frac{C_1}{r} \be_{r} +\frac{C_2}{r} \be_{\theta} 
    =  \frac{1}{x^2+y^2}\begin{pmatrix}
    C_1 x-C_2 y \\
   C_1 y+C_2 x
    \end{pmatrix}, \quad p=- \frac{C_{1}^2+C_{2}^2}{2(x^2+y^2)} .
\end{equation}

\item The flow is purely radial, that is, 
\begin{equation}\label{cdt:g=0}
 g(\theta)\equiv 0.
\end{equation}
In this case, there exists a $2\pi/n$-periodic ($n\in \BN$) elliptic function $f(\theta)$ (see \cite[Theorem 2]{sverak2011landau} and \cite[Section 1]{guillod2015generalized}) such that for any angular shift  $\theta_0\in\mathbb{R}$,
\begin{equation}\label{sveark}
    \bU =\frac{f(\theta-\theta_0)}{r} \be_{r}
    =  \frac{1}{x^2+y^2}\begin{pmatrix}
    f(\theta-\theta_0) x \\
   f(\theta-\theta_0)y
    \end{pmatrix}, \quad 
    p= \frac{2f(\theta-\theta_0)}{x^2+y^2}+\frac{C_1}{x^2+y^2} .
\end{equation}
In addition, the flux $\Phi:=\int_{\mathbb{S}^1}f\,d\theta$ satisfies $4+\Phi/\pi <n^2.$
\end{itemize}
\end{dfn}

\begin{rmk}
Let us make some further comments on Definition \ref{def:JH}.
    \begin{itemize}
        \item  The velocity field $\bU$ in \eqref{sol:CC} leads to an explicit stream function $\psi(x,y)=-C_1 \theta +C_{2} \ln r$ such that $\bU =\nabla^{\perp} \psi$ with $\nabla^{\perp}:=(-\pd_y,\pd_x)^{\top}$,
which has been observed by Jeffery \cite[Equation (III)]{jeffery1915two}. 

\item \cite[Theorem 1.1]{guillod2015generalized} generalized the form of the solutions in \eqref{sveark} under a one-parameter combination of scaling and rotation.

\item  Besides, in the three-dimensional case,  \cite{sverak2011landau} also proved that the nontrivial smooth  self-similar solutions of the stationary Navier--Stokes equations are \emph{Landau solutions}  \cite{landau1944new}.

\end{itemize}
\end{rmk}

The representation \eqref{eq:Up} can also be applied to study the self-similar
flow in the wedge-shaped channel
$$\Omega_{\alpha}=\{(r,\theta):r>0,\ -\alpha<\theta<\alpha\}$$
for some $\alpha\in(0,\pi).$
Assuming that the steady flow satisfies the no-slip boundary conditions on 
plane walls, namely, $\bU=0$ for $\theta=\pm \alpha$, the last equation in \eqref{eq:pfg} implies $g(\theta)\equiv0$ in $\Omega_{\alpha}$. Thus, the steady self-similar flow in $\Omega_{\alpha}$ is purely \emph{radial}.
For radial self-similar flows in $\Omega_{\alpha}$, Dean \cite{dean1934divergent} studied a particular class of divergent flows, while Rosenhead \cite{Rosen1940} carried out a systematic analysis of such flows under the no-slip boundary conditions and expressed the corresponding solutions in terms of elliptic functions. Rosenhead's analysis provided a detailed classification of possible velocity profiles. However, for a prescribed opening angle and flux, the existence of the corresponding solutions was determined primarily through numerical computations. 
Fraenkel \cite{fraenkel1963asymptotic} subsequently employed the self-similar solutions in \cite{Rosen1940} to study the laminar flow through a channel with slightly curved walls.
Furthermore, radial Jeffery--Hamel flows have been also used to analyze practical flows in diverging channels of finite streamwise extent
(see, for instance, \cite{haineshewitthazel2011relation}).

In addition to the works above, let us also recall some results of a more analytical nature in mathematics. For symmetric radial flows in $\Omega_{\alpha}$,
Kobayashi \cite{kobayashi2014jefferyhamel} proved the existence for sufficiently
small flux, subject to restrictions on the opening angle,
and uniqueness within a prescribed bounded class of profiles.
More recently, Bang et al. \cite{bang2024self} studied the existence and uniqueness of different types of flows in terms of the opening angle and flux.

\subsection{Main results}
Inspired by the works \cite{sverak2011landau,guillod2015generalized,bang2024self}, we consider self-similar solutions of \eqref{eq:NS_0} in a cone domain $\Omega \subset \R_{\star}^2$ \emph{without} prescribing conditions on the boundary $\partial\Omega$. Exploiting the symmetry structure \eqref{eq:ss}, 
\eqref{eq:NS_0} can be reduced to a one-dimensional nonlinear ODE given by \eqref{eq:zua-c}, with $\mathcal{F}(z)=0$.
For the \emph{radial} stationary self-similar solutions, the reduced problem \eqref{eq:tilu} can be completely solved, similarly to \cite{sverak2011landau,guillod2015generalized,bang2024self}, from which we obtain a \emph{complete} classification of the local patterns of smooth radial self-similar solutions as in the following theorem.

\begin{thm}[Local profiles of radial self-similar solutions in small cones] \label{thm:main-1}
Let $\Omega \subset \Rs$ be an open cone with vertex at the origin such that 
\begin{equation}\label{cdt:Omega_y}
 \Omega \cap \{(x,y)\in \Rs \mid x=0\} = \emptyset.   
\end{equation}
Let $(u,v,p)\in C^{\infty}(\Omega)^3$ be a nontrivial
local self-similar solution of \eqref{eq:NS_0} in $\Omega$, that additionally satisfies the radial condition
\begin{equation}\label{cdt:xnot0}
  xv = yu, \quad \forall \,\,(x,y)\in\Omega .
\end{equation}
Then $(u,v,p)$ takes the following form
\begin{equation*}
  u(x,y) = \frac{x}{x^2+y^2}\,\kappa, \quad 
  v(x,y) = \frac{y}{x^2+y^2}\,\kappa, \quad
  p(x,y) = \frac{2}{x^2+y^2}\,\kappa + \frac{C_1}{x^2+y^2},
\end{equation*}
where the function $\kappa=\kappa(x,y)$ is given in Table \ref{tab:kap} (see Subsection \ref{subsec:summary}) and $C_1$ is a real constant.
\end{thm}

\begin{rmk}
\label{rmk:main-1}
Let us give some comments on Theorem \ref{thm:main-1}.
\begin{itemize}
        \item We can verify that the condition \eqref{cdt:xnot0} implies $g(\theta)\equiv0$ in $\Omega$ which characterizes the radial self-similar solutions. In this case, the reduced ODE problem (see \eqref{eq:tilu} in Section \ref{sec:symm}) can be integrated in closed form in a similar manner to the discussion in the whole plane \cite{sverak2011landau,guillod2015generalized}. 

    \item The restriction \eqref{cdt:Omega_y} is harmless. Our method also works for the case where $\Omega$ does not intersect with the $x$-axis. In such a situation, one may come across the exactly same ODE problem as in Theorem \ref{thm:main-1} by using another change of variables (see Subsection \ref{subsec:reciprocal}). In this sense, Theorem \ref{thm:main-1} essentially gives all possible smooth branches of the radial self-similar solutions in any small cone region $\Omega.$
    
    \item  
    Compared with previous works on self-similar solutions in wedge-shaped channels,
    Theorem \ref{thm:main-1} also reveals the existence of radial self-similar solutions that does not satisfy the no-slip boundary conditions. As a direct application, we can show that the uniqueness criterion in \cite[Theorem 1.1]{kobayashi2014jefferyhamel} for \emph{symmetric} radial self-similar solutions may fail when the no-slip boundary conditions are relaxed; see Appendix \ref{app:example} for more details.
    
\end{itemize}

\end{rmk}

By Theorem \ref{thm:main-1} and Remark \ref{rmk:main-1}, all the smooth local profiles of  radial self-similar solutions of \eqref{eq:NS_0} have been given whenever the angle $\alpha(\Omega)$ of the cone $\Omega$ is less than $\pi$. As a direct application of Theorem \ref{thm:main-1}, if $\alpha(\Omega) \in (\pi,2\pi],$ then the corresponding smooth local profiles of the radial self-similar solutions of \eqref{eq:NS_0} in $\Omega$ can be regarded as the extension of the ones in Theorem \ref{thm:main-1}. More precisely, we have the following result with a proof given in Section \ref{sec:proof}.

\begin{thm}[Local profiles of radial self-similar solutions in large cones]
\label{thm:main-2}
 Let $\Omega \subset \Rs$ be an open cone with the vertex at the origin and the opening angle 
$\alpha(\Omega)\in(\pi,2\pi].$
Let $(u,v,p)\in C^{\infty}(\Omega)^3$ be a nontrivial
local self-similar solution of \eqref{eq:NS_0} in $\Omega$ additionally satisfying the radial condition
\begin{equation*}
  xv = yu, \quad \forall\,\, (x,y)\in\Omega .
\end{equation*}
Then $(u,v,p)$ takes the following form
\begin{equation}\label{sol:elliptic}
\begin{aligned}
u(x,y) &= \frac{x}{x^2+y^2}\,\wt\kappa(x,y), \quad
v(x,y) = \frac{y}{x^2+y^2}\,\wt\kappa(x,y),\\
p(x,y) &= \frac{2}{x^2+y^2}\,\wt\kappa(x,y) + \frac{C_1}{x^2+y^2},
\end{aligned}
\end{equation}
where $\wt\kappa(x,y)$ is (at most) the real-analytic extension of some $\kappa$ in Theorem \ref{thm:main-1} to $\Omega$, and $C_1\in\R$ is a constant.

Moreover, if the solutions given by \eqref{sol:elliptic} can be smoothly extended to $\Rs,$ then they exactly coincide with the classical radial Jeffery--Hamel solutions in Definition \ref{def:JH}. 
\end{thm}

This result states that the type of local profiles of the smooth self-similar solutions become less as the angle $\alpha (\Omega)$ grows. In particular, we cannot construct the smooth self-similar solution in large cone region by gluing two or more distinct local profiles in the sense of real analyticity. Moreover, the connection between the local profiles in Theorem \ref{thm:main-2} and the radial Jeffery--Hamel solutions in $\Rs$ studied in \cite{sverak2011landau,guillod2015generalized} will be discussed in Subsection \ref{subsec:profiletoJH}.
\medskip

According to Theorems \ref{thm:main-1} and \ref{thm:main-2} above, all smooth
local profiles of radial self-similar solutions of \eqref{eq:NS_0} in a cone $\Omega$ have now
been completely clarified. However, the genuinely non-radial regime is considerably more delicate.
For non-radial flows, the condition \eqref{cdt:xnot0} is replaced by
\begin{equation}\label{cdt:C_0}
   xv = yu + C_0, \quad \forall\,\,(x,y)\in\Omega,
\end{equation}
for some real constant $C_0\neq 0$. 
The following theorem states that some local profiles of the non-radial Jeffery--Hamel solutions are determined by a Li\'enard equation.

\begin{thm}[Existence of non-radial self-similar solutions]
\label{thm:abel-reduction}
Let $\Omega \subset \Rs$ be an open cone with the vertex at the origin satisfying \eqref{cdt:Omega_y}.
Then there exist local self-similar solutions $(u,v,p)\in C^{\infty}(\Omega)^3$ of \eqref{eq:NS_0} in $\Omega$ satisfying the non-radial condition \eqref{cdt:C_0} for some real constant $C_0\neq 0$ as follows: 
\begin{itemize}
   \item Apparently, $(u,v,p)$ can be given as the local profile of the Jeffery--Hamel solution \eqref{sol:CC}. In this case, let us denote $(u,v,p)=(u_{L},v_{L},p_{L})$ with
\begin{equation}\label{eq:U_L}
\begin{aligned}
u_{L}(x,y):= \frac{\wt C_1 x-C_0 y}{x^2+y^2},\quad 
v_{L}(x,y) := \frac{ \wt C_1 y +C_0 x}{x^2+y^2},\quad
p_{L}(x,y) := -\frac{\wt C_1^2+C_0^2}{2(x^2+y^2)}
\end{aligned}
\end{equation}
for any real number $\wt C_1$ and $(x,y)\in \Omega.$
\item 
On the other hand, $(u,v,p)$ given by
\begin{equation*}
\begin{aligned}
u(x,y) &= u_{L}(x,y)+\frac{x}{x^2+y^2}\,H\big(\arctan \,(y/x)\big), \\
v(x,y) &= v_{L}(x,y)+\frac{y}{x^2+y^2}\,H\big(\arctan \,(y/x)\big),\\
p(x,y) &= p_{L}(x,y)+\frac{2}{x^2+y^2}\,H\big(\arctan \,(y/x)\big)
\end{aligned}
\end{equation*}
is also a self-similar solution of \eqref{eq:NS_0} in $\Omega$, if $H=H(\theta)$ ($\theta=\arctan\,(y/x)$) is a nonconstant smooth solution of the Li\'enard equation (see also equation \eqref{eq:Heq})
\begin{equation}\label{eq:H_pm}
\frac{d^2 H}{d\theta^2}-C_{0}\frac{d H}{d\theta} + H^{2} + (2 \wt C_1+4)\,H=0.
\end{equation}
\end{itemize}
\end{thm}

Unlike the radial self-similar solutions, the ODE \eqref{eq:H_pm} is not generally integrable. Under the integrability condition \eqref{eq:CC1},  the Li\'enard equation becomes a Painlev\'e--Gambier type-II equation, yielding a non-radial self-similar solution of the SNS equations \eqref{eq:NS_0} expressed in terms of Weierstrass elliptic functions, as shown in \eqref{eq:uuuvvv}. In the degenerate case, this solution reduces to \eqref{eq:uv-special21-exp}, which is given in terms of elementary functions. Further details are provided in Section \ref{sec:nonradial}.

\section{Symmetries and radial self-similar solutions}
\label{sec:symm}

Exact solutions of the Navier--Stokes equations and classifications of flows according to their geometry have been studied extensively; see, e.g., \cite{drazinriley2006navier}. Symmetry reductions of the Navier--Stokes equations have also been studied systematically; see, e.g., \cite{fushchych1994symmetry}. In particular, self-similar solutions can be obtained through symmetry reductions associated with scaling transformations such as \eqref{eq:scasym}. For a general treatment of symmetry methods for differential equations, we refer the reader to \cite{olver1993,bluman2010applications}.

In this section, we mainly focus on the (Lie point) symmetries of SNS equations \eqref{eq:NS_0},
which can be equivalently written as
\begin{equation}
\label{eq:0forceNS}
	\begin{cases}
	-u_{xx} - u_{yy} + uu_x + vu_y + p_x = 0, \\
		-v_{xx} - v_{yy} + uv_x + vv_y + p_y = 0, \\
		u_x + v_y = 0.
	\end{cases}
\end{equation}
These symmetries, in particular the scaling symmetry, will be used to construct the corresponding self-similar solutions. 
Here, we used $u_x=\pd_x u$, $u_{xy}=\pd_{x}\pd_{y}u$ and so on to denote partial derivatives. 

Symmetries of the equations \eqref{eq:0forceNS} are local diffeomorphisms 
\begin{equation*}
   (x,y,u,v,p) \mapsto (\widehat{x}, \widehat{y}, \widehat{u},\widehat{v},\widehat{p}),
\end{equation*}
which depend on a real parameter $\varepsilon$  smoothly around the identity transformation for $\varepsilon=0$, and transform one solution to another solution (or to itself). To obtain symmetries,  it suffices to calculate the corresponding infinitesimal generators 
	\begin{equation*}
	\mathbf{v} = \xi \partial_x + \eta \partial_y + \phi \partial_u + \psi \partial_v + \chi  \partial_p
	\end{equation*}
 by using the linearized symmetry condition \cite{olver1993},
 where all coefficients are functions of $(x,y,u,v,p)$ defined by
 \begin{equation*}
 \xi=\frac{\operatorname{d}\!}{\operatorname{d}\!\varepsilon} \Big|_{\varepsilon=0}\widehat{x},\quad  \eta=\frac{\operatorname{d}\!}{\operatorname{d}\!\varepsilon} \Big|_{\varepsilon=0}\widehat{y},\quad  \phi=\frac{\operatorname{d}\!}{\operatorname{d}\!\varepsilon} \Big|_{\varepsilon=0}\widehat{u},\quad  \psi=\frac{\operatorname{d}\!}{\operatorname{d}\!\varepsilon} \Big|_{\varepsilon=0}\widehat{v},\quad  \chi=\frac{\operatorname{d}\!}{\operatorname{d}\!\varepsilon} \Big|_{\varepsilon=0}\widehat{p}.
 \end{equation*}

 Solving the linearized symmetry condition for the system \eqref{eq:0forceNS}, we obtain
 	\begin{equation*}
		\xi = c_1 + c_3x + c_4y, \quad 
		\eta = c_2 + c_3y - c_4x, \quad
		\phi = -c_3u + c_4v, \quad 
		\psi = -c_3v - c_4u, \quad
		\chi = -2c_3p+c_5,
	\end{equation*}
 where $c_1,c_2,c_3,c_4,c_5$ are constants. Namely, symmetries of \eqref{eq:0forceNS} are generated by the following infinitesimal generators
\begin{equation}\label{eq:sym}
     \begin{aligned}
         \mathbf{v}_1 &= \partial_x, \quad  &\text{translational symmetry along } x,\\
         \mathbf{v}_2 & = \partial_y, \quad  &\text{translational symmetry along } y,\\
          \mathbf{v}_3  & = \partial_p, \quad  &\text{translational symmetry along } p,   \\
         \mathbf{v}_4 & = x\partial_x+y\partial_y-u\partial_u-v\partial_v-2p\partial_p, \quad &\text{scaling symmetry}, \\
         \mathbf{v}_5 & = -y\partial_x + x\partial_y - v\partial_u + u\partial_v, \quad &\text{rotational symmetry}.
     \end{aligned}
 \end{equation}

\subsection{Reduction of the general forced stationary problem in Cartesian coordinates}

Before addressing the unforced problem \eqref{eq:0forceNS}, let us remark briefly on the extension of the symmetries \eqref{eq:sym} to the forced one:
\begin{equation}
\label{eq:NS}
	\begin{cases}
	-u_{xx} - u_{yy} + uu_x + vu_y + p_x = F_1, \\
		-v_{xx} - v_{yy} + uv_x + vv_y + p_y = F_2, \\
		u_x + v_y = 0,
	\end{cases}
\end{equation}
where $\bF(x,y)=(F_1(x,y),F_2(x,y))$ denotes the given external force.

In the current paper, we will only consider the scaling symmetry generated by the infinitesimal generator $\mathbf{v}_4$ in \eqref{eq:sym}. The corresponding transformation now reads
\begin{equation*}
    \widehat{x}=x\exp\,(\varepsilon), \quad \widehat{y}=y\exp\,(\varepsilon),\quad 
    \widehat{u}=u\exp\,(-\varepsilon),\quad \widehat{v}=v\exp\,(-\varepsilon),\quad 
    \widehat{p}=p\exp\,(-2\varepsilon).
\end{equation*}
Consequently, the terms on the left-hand side of the first two equations in \eqref{eq:NS} are homogeneous of order $-3$.  Therefore, the system \eqref{eq:NS} is again scaling invariant if $\mathbf{F}$ is homogeneous of order $-3$, i.e., 
$$\mathbf{F}(\lambda x,\lambda y)=\lambda^{-3}\mathbf{F}(x,y), \quad  \lambda\neq0.$$

In the cone-shaped domain $\Omega \subset \Rs$ with vertex at the origin and \emph{not intersecting with the $y$-axis}, the invariants with respect to $\mathbf{v}_4$ can be chosen as
\begin{equation}\label{eq:invs}
    z:=y/x,\quad 
    \widetilde{u}(z):=xu(x,y),\quad 
    \widetilde{v}(z):=xv(x,y), \quad 
    \widetilde{p}(z):=x^2p(x,y)
\end{equation}
for any $(x,y)\in \Omega$, and we define
\begin{equation*}
        \wt{\mathbf{F}}(z)=x^3\mathbf{F}(x,y) =\mathbf{F}(1,z).
\end{equation*}
Then using \eqref{eq:NS}, the new variables $\widetilde{u}$, 
$\widetilde{v}$, $\widetilde{p}$ and $\widetilde{\bF}$ satisfy the following reduced ODEs:
\begin{equation} \label{eq:redsnsfor}
    \begin{cases}
    (z^2 + 1)\widetilde{u}_{zz} + z\widetilde{u}\widetilde{u}_z - \widetilde{u}_z\widetilde{v} + 4z\widetilde{u}_z + \widetilde{u}^2 + 2\widetilde{u} + z\widetilde{p}_z + 2
	\widetilde{p} = -\wt F_1(z), \vspace{0.1cm}\\
	(z^2 + 1)\widetilde{v}_{zz} + z\widetilde{u}\widetilde{v}_z - \widetilde{v}_z\widetilde{v} + 4z\widetilde{v}_z + \widetilde{u}\widetilde{v} + 2\widetilde{v} - \widetilde{p}_z = -\wt F_2(z), \vspace{0.1cm}\\
	z\widetilde{u}_z + \widetilde{u} - \widetilde{v}_z = 0.
    \end{cases}
\end{equation}
The last equation in \eqref{eq:redsnsfor} gives 
\begin{equation}\label{C_0-neq-0}
    \widetilde{v} = z\widetilde{u} +C_0
\end{equation}
for some real number $C_0$. 
Now, we substitute \eqref{C_0-neq-0} into \eqref{eq:redsnsfor} to get 
\begin{subequations}
\begin{empheq}[left=\empheqlbrace]{align}
 &(z^2+1)\widetilde{u}_{zz} + (4z-C_0)\widetilde{u}_z
 + \widetilde{u}^2 + 2\widetilde{u}
 + z\widetilde{p}_z + 2\widetilde{p}
 = -\wt F_1(z),\label{eq:zua_a}  \vspace{0.1cm} \\ 
&(z^3+z)\widetilde{u}_{zz}
 + (-C_0 z + 6z^2 + 2)\widetilde{u}_z
 + z\widetilde{u}^2 + 6z\widetilde{u}
 + 2C_0 - \widetilde{p}_z
 = -\wt F_2(z).\label{eq:zua_b}
\end{empheq}
\end{subequations}
Multiplying \eqref{eq:zua_b} by $z$ and adding the result to \eqref{eq:zua_a} eliminates the term $z\widetilde{p}_z$, yielding
\begin{equation}\label{eq:zub+a}
  \begin{aligned}
    \widetilde{p}=-\frac{1}{2}\Big( (z^2+1)^2\widetilde{u}_{zz} &+( 6z-C_0)(z^2+1)\widetilde{u}_z+(z^2+1)\widetilde{u}^2\\
    &+ (6z^2+2)\widetilde{u}+2C_0z+\wt{F}_1(z)+z\wt{F}_2(z)\Big).
\end{aligned} 
\end{equation}

On the other hand, multiplying \eqref{eq:zua_a} by $z$ and subtracting \eqref{eq:zua_b} eliminates the term $\widetilde{u}_{zz}$, leading to
\begin{equation*}
\frac{d}{dz} \left( (z^2+1)\widetilde{p}-2(z^2+1)\widetilde{u}-   2C_0z\right)=\wt{F}_2(z)-z\wt{F}_1(z).
\end{equation*}
Integrating this equation yields
\begin{equation}\label{eq:zua-b}
    \widetilde{p}=2\widetilde{u} +\frac{2C_0z}{z^2+1}+\frac{1}{z^2+1}
    \left( C_1+\int_{z_0}^z \left(  \wt{F}_2(w)-w\wt{F}_1(w)\right)dw \right)
\end{equation}
for some real numbers $C_1$ and $z_0$. 

At last, equating \eqref{eq:zub+a} and \eqref{eq:zua-b}, we get the following reduced ODE of $\wt{u}:$
\begin{equation}\label{eq:zua-c}
   \begin{split}
& (z^2+1)^2\widetilde{u}_{zz} + (6z-C_0)(z^2+1)\widetilde{u}_{z} 
+ (z^2+1)\widetilde{u}^2 + 6(z^2+1)\widetilde{u} \\
&\hspace{5.5cm}+\frac{2C_0(z^3+3z)}{z^2+1}
+\frac{2C_1}{z^2+1} + \mathcal{F}(z)=0   
\end{split}
\end{equation}
where we have set 
\begin{equation*}
    \mathcal{F}(z):=\wt{F}_1(z)+z\wt{F}_2(z)
    +\frac{2}{z^2+1}\int_{z_0}^z \left(  \wt{F}_2(w)-w\wt{F}_1(w)\right)dw.
\end{equation*}

\subsection{Radial self-similar solutions: Scaling symmetry}
\label{subsec:solve}
In this subsection, we will calculate radial self-similar solutions of the unforced SNS equations \eqref{eq:0forceNS} in a cone-shaped domain $\Omega \subset \Rs$ with vertex at the origin and not intersecting with the $y$-axis. 
 Taking $\wt{\bF}=0$ and $C_0=0$ in \eqref{eq:zua-b} and \eqref{eq:zua-c},  the pressure law and the corresponding reduced ODE of \eqref{eq:0forceNS} respectively read
\begin{equation*}
    \widetilde{p}=2\widetilde{u} +\frac{C_1}{z^2+1}
\end{equation*}
and 
\begin{equation}
\label{eq:tilu}
(z^2+1)^2\widetilde{u}_{zz} + 6z(z^2+1)\widetilde{u}_{z} 
+ (z^2+1)\widetilde{u}^2 + 6(z^2+1)\widetilde{u} 
+ \frac{2C_1}{z^2+1}= 0
\end{equation}
for some real constant $C_1$.

To solve \eqref{eq:tilu}, let us define 
\begin{equation*}
  h := (z^2+1) \widetilde{u},
\end{equation*}
from which we have  
\begin{equation}\label{sol:3}
    u(x,y)= \frac{\widetilde{u}(z)}{x} = \frac{h(z)}{x(z^2+1)}=\frac{x}{x^2+y^2} h\left(\frac{y}{x}\right).
\end{equation}
Hence, it suffices to derive the formula of $h$ instead of $\wt{u}$.
\medskip

Now, we substitute $\widetilde{u}=(z^2+1)^{-1}h$ into \eqref{eq:tilu} and obtain 
\begin{equation}\label{h_1}
(z^2+1)^2h_{zz}+2z(z^2+1)h_z+h^2+4h+2C_1=0.
\end{equation}
Multiplying both sides of \eqref{h_1} by $2h_z$ and integrating it, we obtain
\begin{equation*}
(z^2+1)^2h_z^2+\frac{2}{3}h^3+4h^2+4C_1h+\frac{2}{3} C_2=0   
\end{equation*}
for some real constant $C_2$, which implies 
\begin{equation}\label{I_0}
  h_z =\pm \frac{\sqrt{6}}{3}  \frac{\sqrt{-h^3-6h^2-6C_1h-C_2}}{z^2+1}
\end{equation}
so long as 
\begin{equation*}
 \CP_3(h):= h^3+6h^2+6C_1h+C_2 \leq 0.  
\end{equation*}
 \medskip
 
First of all, a \emph{singular solution} can be obtained by assuming  $\CP_3(h(z))\equiv 0$ for any $z$. 
In this situation, it is not hard to see that  
\begin{equation*}
 h (z)\equiv C   
\end{equation*}
for some real number $C$ satisfying $\CP_3(C)=0$. 
Therefore, the unforced problem \eqref{eq:0forceNS} admits the following solution:
\begin{equation}\label{eq:sinsol}
  u(x,y)=\frac{Cx}{x^2+y^2}, \quad 
  v(x,y)= \frac{Cy}{x^2+y^2},\quad 
  p(x,y)=  -\frac{C^2}{2(x^2+y^2)}
\end{equation}
in view of the equalities \eqref{cdt:xnot0}, \eqref{eq:zub+a} and \eqref{sol:3}.
Moreover, such $C$ in \eqref{eq:sinsol} can  run throughout the real line by inserting \eqref{eq:sinsol} into \eqref{eq:NS_0}.

Next, assuming $\CP_3(h)<0$,
the general solution of \eqref{I_0} can be written in the quadrature form
\begin{equation}\label{eq:I_C_1}
I:=\int_{h_0}^h  \frac{d \wt{h}}{\sqrt{-\wt{h}^3-6\wt{h}^2-6C_1\wt{h}-C_2}}=\pm \frac{\sqrt{6}}{3}\arctan\, z+C_{3} 
\end{equation}
for some (suitable) real constants $h_0$ and $C_{3}$.  
Let $\Delta_{\rm cubic}$ denote the discriminant of the following cubic equation 
\begin{equation}\label{eq:cubic_C_1}
\CP_3(h)=h^3+6h^2+6C_1h+C_2=0, 
\end{equation}
namely, 
$$\Delta_{\rm cubic}:=-27\left( C_2^2+(-24C_1+32)C_2+32C_1^3-48C_1^2\right).$$
In addition, we regard the form 
\begin{equation*}
   \CQ_2(C_2):= -C_2^2+(24C_1-32)C_2-32C_1^3+48C_1^2
\end{equation*}
as a quadratic polynomial for $C_2$, whose discriminant is
 $$\Delta_{\rm square}:=-128 (C_1-2)^3.$$
 
 The solvability of \eqref{eq:I_C_1} is divided into the following several cases depending on the value of the parameters $C_1$ and $C_2$. However, most of the solutions of \eqref{eq:I_C_1} cannot be formulated by elementary functions. In fact, we will mainly use elliptic integrals. More precisely, $\mathrm{am}\,(m, k)$ denotes the \emph{Jacobi amplitude function} with argument $m$ and modulus $k$, which is the inverse of the elliptic integral of the first kind satisfying 
\begin{equation}\label{eq:am}
    m = \int_0^{\mathrm{am}\,(m,k)} \frac{d\theta}{\sqrt{1 - k^2 \sin^2\theta}}
    =F\left(\mathrm{am}\,(m,k),k\right),
\end{equation}  
and  $K(k)=F(\pi/2,k)$ denotes the \emph{complete elliptic integral of the first kind};
 the \emph{incomplete elliptic integral of the first kind} is defined as
\begin{equation*}
F(\phi,k):=\int^{\phi}_{0}\frac{1}{\sqrt{1-k^2\sin^2\theta}}\,d\theta.
\end{equation*}

{\bf Case 1.} Suppose $C_1>2$. In this case, we have 
\begin{equation*}
    \Delta_{\rm square}<0 \quad \text{and}\quad \Delta_{\rm cubic}<0,
\end{equation*}
which implies the equation \eqref{eq:cubic_C_1} has two complex conjugate roots $m \pm n i$ and one real root $\alpha$ for $\alpha,m,n\in \mathbb{R}$. Then equating 
\begin{equation*}
\CP_3(h)=(h-\alpha)(h-m - ni)(h-m + ni)
\end{equation*}
provides us with the following conditions:
\begin{equation}\label{eq:mn}
    \begin{aligned}
    n\neq 0,\quad -2m-\alpha=6,\quad m^2+n^2+2\alpha m=6C_1,\quad -\alpha(m^2+n^2)=C_2.
    \end{aligned}
\end{equation}

Now, for any\footnote{Indeed, we may firstly take $h_0=-N$ and then verify the limit as $N$ tends to the infinity. In this sense, we put $h_0=-\infty$ if there is no confusion.}
$-\infty=h_0<h< \alpha$, we introduce the change of variables
    \begin{equation*}
     \theta = {\rm \,arccot\,}  \sqrt{\frac{\alpha - \wt{h}}{\beta}}
     \quad \text{for}\,\,\,0<\theta<\frac{\pi}{2}
    \end{equation*}
with $\beta :=\sqrt{(m-\alpha)^2+n^2}>0$.
Immediately, we have 
\begin{equation*}
    d\wt{h} = 2\beta  \frac{\cos\theta}{\sin^3\theta} \, d\theta, \quad 
\alpha- \wt{h} = \beta \cot^2\theta.
\end{equation*}    
Moreover, note that 
\begin{align*}
(\wt{h}-m)^2+n^2 &=\beta^2 (1+\cot^4 \theta) +2 (m-\alpha)\beta \cot^2 \theta\\
&=\frac{\beta^2}{\sin^4 \theta}- 2 (\beta-m+\alpha)\beta \cot^2 \theta\\
&= \frac{\beta^2}{\sin^4 \theta} \left( 1- 2 (\beta-m+\alpha)\beta^{-1}  \sin^2 \theta \cos^2 \theta  \right),
\end{align*}
which provides us with 
\begin{equation}\label{eq:I_1}
 \begin{aligned}
    I &= \int_{-\infty}^{h} \frac{d\wt{h}}{\sqrt{(\alpha-\wt{h})\left( (\wt{h}-m)^2+n^2 \right)}} \\
    &= \frac{2}{\sqrt{\beta}}
    \int_{0}^{{\rm \,arccot\,}  \sqrt{\frac{\alpha - h}{\beta}}} 
    \frac{d\theta}{\sqrt{1- 2 (\beta-m+\alpha)\beta^{-1}  \sin^2 \theta \cos^2 \theta }} \\
     &=\frac{1}{\sqrt{\beta}}F\left(2 {\rm\, arccot\,} \sqrt{\frac{\alpha-h}{\beta}},\sqrt{\frac{\beta-m+\alpha}{2\beta}}\,\right).
\end{aligned}   
\end{equation}
Then combining \eqref{eq:I_C_1} and \eqref{eq:I_1} yields
\begin{equation}\label{eq:hhh}
    h(z) = \alpha - \beta \cot^2\left( \frac{1}{2} \mathrm{am}\left(   \frac{\sqrt{6\beta}}{3} \arctan \,z +\sqrt{\beta}C_{4}, \sqrt{\frac{\beta - m + \alpha}{2\beta}} \,\,\right) \right)
\end{equation}
for some real constant $ C_{4}\in\mathbb{R}$. 
\medskip

{\bf Case 2.} Suppose $C_1=2$, i.e., $\Delta_{\rm square}=0$ and $\Delta_{\rm cubic}=-27(C_2-8)^2$. Then the calculation of the integral $I$ depends on the choice of $C_2$.
\begin{itemize}
    \item If $C_2 =8$,  then $\Delta_{\rm cubic}=0$, and \eqref{eq:cubic_C_1} has three identical real roots $-2$, that is, 
\begin{equation*}
    \CP_3(h)=(h+2)^3=0.
\end{equation*}
Thus, it is easy to see that
\begin{equation*}
 I=\int_{-\infty}^h \frac{d\wt{h}}{\sqrt{-(\wt{h}+2)^3}} = \frac{2}{\sqrt{-(h+2)}}  
\end{equation*}
for any $-\infty=h_0<h<-2$, which, together with \eqref{eq:I_C_1}, yields 
\begin{equation*}
    h(z)= -2 \,- \,\frac{6}{\left(  \arctan\, z + C_{4} \right)^2}
\end{equation*}
for some $C_{4}\in\mathbb{R}$.

    \item  If $C_2 \neq 8$, then $\Delta_{\rm cubic}<0$. Thus, analogous to {\bf Case 1}, the equation \eqref{eq:cubic_C_1} has two complex conjugate roots $m \pm n i$ and one real root $\alpha$ satisfying 
\begin{equation*}
    \begin{aligned}
        n\neq 0, \quad -2m-\alpha=6,\quad 
        m^2+n^2+2\alpha m=12,\quad -\alpha(m^2+n^2)=C_2.
    \end{aligned}
\end{equation*}
Furthermore, we can solve \eqref{eq:I_C_1} in a similar manner and obtain 
the same form \eqref{eq:hhh} with the parameters $C_1=2$ and $C_2\neq 8$.
\end{itemize}
\medskip

{\bf Case 3.} Suppose $C_1 <2$ and then $\Delta_{\rm square}>0$. However, the value of $\Delta_{\rm cubic}$ can be positive, zero, or negative.

\begin{itemize}
    \item Note that $\Delta_{\rm cubic}>0$ whenever $$ 12C_1-16-4(2-C_1)\sqrt{2(2-C_1)}<C_2<12C_1-16+4(2-C_1)\sqrt{2(2-C_1)}.$$ 
Now, the equation \eqref{eq:cubic_C_1} has three distinct real roots $a<b<c$. 
Then the integral
\begin{equation*}
    I=\int_{h_0}^{h} \frac{d\wt{h}}{\sqrt{-(\wt{h}-a)(\wt{h}-b)(\wt{h}-c)}}
\end{equation*}
is valid for $b<h<h_0\leq c$ or $h<h_0\leq a$.

\begin{itemize}
    \item If $b<h<h_0\leq c$,  then we introduce the change of variables
    \begin{equation*}
     \theta = \arcsin \sqrt{\frac{c - \wt{h}}{c - b}} \quad \text{for}\,\,\,0<\theta<\frac{\pi}{2}.
    \end{equation*}
Now, taking $h_0=c$ and  using the facts that
\begin{align*}
d\wt{h} &= -2(c - b)\sin\theta\cos\theta \, d\theta, &
c - \wt{h} &= (c - b)\sin^2\theta, \\
\wt{h} - a &= (c - a) - (c - b)\sin^2\theta, &
\wt{h} - b &= (c - b)\cos^2\theta,
\end{align*}
we can rewrite $I$ in the form of an elliptic integral as follows
\begin{equation*}
\begin{aligned}
I&=-\int_{h}^{c} \frac{d\wt{h}}{\sqrt{(\wt{h} - a)(\wt{h} - b)(c - \wt{h})}} \\
&= -2\int^{\arcsin\left( \sqrt{\frac{c-h}{c-b}}\right)}_{0} \frac{d\theta}{\sqrt{(c - a) - (c - b)\sin^2\theta}}\\
&=\displaystyle -\frac{2}{\sqrt{c-a}} 
F \left( \arcsin \sqrt{\frac{c-h}{c-b}}, \sqrt{\frac{c-b}{c-a}}\, \right).
\end{aligned}
\end{equation*}
Then \eqref{eq:I_C_1} and \eqref{eq:am} amount to
\begin{equation*}
    \begin{aligned}
        h(z)&=\displaystyle c - (c - b) \sin^2\left( \mathrm{am}\left( 
        \sqrt{\dfrac{c-a}{6}} \arctan\, z +\sqrt{c-a}\, C_{4}, \sqrt{\frac{c-b}{c-a}} \,\right) \right)
    \end{aligned}
\end{equation*}
for some $C_{4}\in\mathbb{R}$. 

\item 
If $h<h_0\leq a$, then we use the change of variables 
\begin{equation*}
\theta = \arctan\, \sqrt{\frac{a - \wt{h}}{b -a }}
 \quad \text{for}\,\,\,   0<\theta<\frac{\pi}{2},
\end{equation*}
which implies that 
\begin{align*}
d\wt{h} &=-\frac{2(b-a)\sin\theta}{\cos^3\theta} \, d\theta, &
a - \wt{h} &= (b - a)\tan^2\theta, \\
b - \wt{h} &= \frac{b - a}{\cos^2\theta}, &
c - \wt{h} &= (c-a)+ (b-a)\tan^2\theta.
\end{align*}
Consequently, we have 
\begin{align*}
I&=-\int_{h}^{a} \frac{d \wt{h}}{\sqrt{(a -\wt{h})(b - \wt{h})(c - \wt{h})}} \\
&= -2\int_{0}^{\arctan\, \left( \sqrt{\frac{a-h}{b-a}}\right)} \frac{d\theta}{\sqrt{(c - a) - (c - b)\sin^2\theta}}\\
&=\displaystyle -\frac{2}{\sqrt{c-a}} F \left( \arctan\, 
         \sqrt{\frac{a-h}{b-a}}, \sqrt{\frac{c-b}{c-a}} \right).
\end{align*}
Finally, \eqref{eq:I_C_1} and \eqref{eq:am} imply that
\begin{equation*}
    \begin{aligned}
        h(z)&=\displaystyle a - (b - a) \tan^2\left( \mathrm{am}\left( 
        \sqrt{\dfrac{c-a}{6}}  \arctan\, z +\sqrt{c-a}\,C_{4}, \sqrt{\frac{c-b}{c-a}} \,\right) 
        \right)
    \end{aligned}
\end{equation*}
for some $C_{4}\in\mathbb{R}$. 
\end{itemize}

\item If we take
\begin{equation*}
  C_2 =12C_1-16+4(2-C_1)\sqrt{2(2-C_1)}  
\end{equation*}
or 
\begin{equation*}
C_2=12C_1-16-4(2-C_1)\sqrt{2(2-C_1)},
\end{equation*}
then $\Delta_{\rm cubic}=0$ and the equation \eqref{eq:cubic_C_1} has three real roots and two of them are identical.

\begin{itemize}
    \item If $C_2=12C_1-16+4(2-C_1)\sqrt{2(2-C_1)}$, then we see that
\begin{align*}
   \CP_3(h)&=\left(h+2+2\sqrt{2(2-C_1)}\right)\left(h+2-\sqrt{2(2-C_1)}\right)^2\\
&   =(h-a)(h-b)^2
\end{align*}
for
\begin{equation*}
 a:=-2-2\sqrt{2(2-C_1)}< b:=-2+\sqrt{2(2-C_1)}.
\end{equation*}
Therefore, 
\begin{align*}
I&=-\int_{h}^a \frac{d\wt{h}}{(-\wt{h}+b)\sqrt{-\wt{h}+a}} \\
&=-\frac{2}{\sqrt{b-a}}\arctan\,\sqrt{\frac{a-h}{b-a}},
\end{align*}
for $h<h_0:=a$, and  \eqref{eq:I_C_1} gives
\begin{equation*}
   h(z)=\displaystyle -2-2\sqrt{2(2-C_1)}  
-3\sqrt{2(2-C_1)} \tan^2\left(\sqrt[4]{\frac{2-C_1}{2}}\arctan\, z + \sqrt[4]{2-C_1} C_{4} \right)   
\end{equation*}
for some $C_{4}\in\mathbb{R}$. 

\item 
If $C_2=12C_1-16-4(2-C_1)\sqrt{2(2-C_1)}$, then we have
\begin{align*}
      \CP_3(h)=\left(h+2-2\sqrt{2(2-C_1)}\right)\left(h+2+\sqrt{2(2-C_1)}\right)^2
     =(h-b)(h-a)^2,
\end{align*}
where, in this case, we have again
\begin{equation*}
a:=-2-\sqrt{2(2-C_1)} <
b:=-2+2\sqrt{2(2-C_1)}.
\end{equation*}
By direct calculations, we have 
\begin{equation*}
\begin{aligned}
I&=
\begin{cases}
\displaystyle-\int_{h}^{b} \frac{d\wt{h}}{(\wt{h}-a)\sqrt{-\wt{h}+b}}
& \text{if} \,\,\, a<h<h_0=b; \vspace{0.2cm}\\\
\displaystyle\int_{-\infty}^h \frac{d\wt{h}}{(-\wt{h}+a)\sqrt{-\wt{h}+b}} 
& \text{if} \,\,\, -\infty=h_0<h<a,
\end{cases}
\\&=\begin{cases}
     \displaystyle-\frac{1}{\sqrt{b-a}} \ln \frac{\sqrt{b-a}+\sqrt{b-h}}{\sqrt{b-a}-\sqrt{b-h}}
     & \text{if} \,\,\, a<h<h_0=b; \vspace{0.25cm}\\\
      \displaystyle \frac{1}{\sqrt{b-a}} \ln \frac{\sqrt{b-a}+\sqrt{b-h}}{\sqrt{b-h}-\sqrt{b-a}}  
      & \text{if} \,\,\, -\infty=h_0<h<a.
\end{cases}
\end{aligned}
\end{equation*}
Substituting the $I$ above into \eqref{eq:I_C_1} yields
\begin{equation*}
    \begin{aligned}
        \displaystyle h(z) &=-2 + 2\sqrt{2(2 - C_1)} - 3\sqrt{2(2 - C_1)} \cdot \tanh^2\left( \frac{\sqrt[4]{2 - C_1}}{2}  \left( \sqrt[4]{8} \arctan\,\left(z\right) + C_4 \right) \right)\
     \end{aligned}      
\end{equation*}
if $ a<h<h_0=b$, or 
\begin{equation*}
    \begin{aligned}
            \displaystyle h(z)&= -2 + 2\sqrt{2(2 - C_1)} - 3\sqrt{2(2 - C_1)} \cdot \coth^2\left( \frac{\sqrt[4]{2 - C_1}}{2} \left( \sqrt[4]{8} \arctan\,\left(z\right) + C_4 \right) \right)\
     \end{aligned}      
\end{equation*}
if $-\infty=h_0<h<a$, with $C_{4}\in\mathbb{R}$. 

\end{itemize}

\item  At last, we assume that 
\begin{equation}\label{cdt:C12}
 C_2 >12C_1-16+4(2-C_1)\sqrt{2(2-C_1)}~~ \text{or}~~ C_2<12C_1-16-4(2-C_1)\sqrt{2(2-C_1)}.   
\end{equation}
Then $\Delta_{\rm cubic}<0$ and \eqref{eq:cubic_C_1} has two conjugate complex roots and one real root. This case is again similar to {\bf Case 1}. The solution is of the form \eqref{eq:hhh} with the parameters satisfying \eqref{eq:mn}, but $C_1$ and $C_2$ satisfy $C_1<2$ 
and \eqref{cdt:C12}.
\end{itemize}

\subsection{Summary and further comments}
\label{subsec:summary}
Summarizing the discussion above, we obtain the corresponding self-similar solutions of the SNS equations \eqref{eq:0forceNS} straightforwardly (see also Theorem \ref{thm:main-1}). More precisely, the similar solutions are written in the form:
\begin{equation*}
 u(x,y) = \displaystyle\frac{x}{x^2+y^2}\kappa(x,y),\quad 
 v(x,y) = \displaystyle\frac{y}{x^2+y^2}\kappa(x,y),\quad 
\end{equation*}
where the formula of the function $\kappa(x,y):=h(y/x)$ is summarized below.

As shown in Figure \ref{fig:1}, the parameter plane $(C_1,C_2)$ is divided into several regions. 
For brevity, we denote
\begin{equation}\label{eq:domain}
 \begin{aligned}
 \text{I} &:= \text{I}_1 \cup\Gamma_0\cup\text{I}_2,\\
        P_0&:=\{(2,8)\},\\
         \text{II}&:=\{(C_1,C_2)\in \R^2:C_1<2, L_-(C_1)<C_2 <L_+(C_1)\},\\
        \Gamma_+&:=\{(C_1,C_2) \in \R^2: C_1<2, C_2=L_+(C_1)\},\\
        \Gamma_-&:=\{(C_1,C_2) \in \R^2: C_1<2, C_2=L_-(C_1)\},
\end{aligned}   
\end{equation}
where
\begin{equation*}
\begin{aligned}
 \text{I}_1 &:= \{(C_1,C_2)\in \R^2:C_1>2\},\\
        \Gamma_0 &:=\{(2,C_2) \in \R^2: C_2 \neq 8\}, \\ 
        \text{I}_2&:=\{(C_1,C_2)\in \R^2:C_1<2,C_2> L_+(C_1) \,\,\,\text{or}\,\,\, C_2<L_-(C_1)\},\\
     L_\pm(C_1)&:=12C_1-16\pm 4(2-C_1)\sqrt{2(2-C_1)}.
     \end{aligned}
\end{equation*}

\begin{figure}[H]
\centering
\begin{tikzpicture}
\begin{axis}[
    axis lines = middle,
    xlabel = $C_1$,
    ylabel = $C_2$,
    xmin = -2, xmax = 3,  
    ymin = -35, ymax = 20,
    domain = -2:2,         
    samples = 200,
    width=10cm,              
    height=10cm,
    clip mode=individual,
    enlargelimits=false,
    axis line style={thick}, 
    label style={font=\large}, 
    tick label style={font=\large} 
]

\fill[pattern=north east lines, pattern color=gray!30] 
    (2,\pgfkeysvalueof{/pgfplots/ymin}) rectangle (3,\pgfkeysvalueof{/pgfplots/ymax});
\node[anchor=center] at (axis cs: 2.5,-10) {\Huge\bfseries $\text{I}_1$};

\draw[dashed, thick] (2,\pgfkeysvalueof{/pgfplots/ymin}) -- (2,\pgfkeysvalueof{/pgfplots/ymax});

\addplot[name path=upper, thick, smooth, color=blue!80!black] {12*x -16 + 4*(2 - x)*sqrt(2*(2 - x))};
\addplot[name path=lower, thick, smooth, color=red!80!black] {12*x -16 - 4*(2 - x)*sqrt(2*(2 - x))};

\addplot[pattern=horizontal lines, pattern color=green!30] 
    fill between[of=upper and lower, soft clip={domain=-2:2}];
\node[anchor=center] at (axis cs: 0.5,-8) {\Huge\bfseries II};

\path[name path=axis_top] (-2,20) -- (2,20);
\path[name path=axis_bottom] (-2,-35) -- (2,-35);
\addplot[pattern=vertical lines, pattern color=orange!30]
    fill between[of=upper and axis_top];
\addplot[pattern=vertical lines, pattern color=orange!30]
    fill between[of=lower and axis_bottom];
\node[anchor=center] at (axis cs: 1,12) {\Huge\bfseries $\text{I}_2$}; 
\node[anchor=center] at (axis cs: 1.3,-25) {\Huge\bfseries $\text{I}_2$};

\node[circle, fill=red, inner sep=3pt, label={[red,font=\Large]above right:$P_0$}] at (2,8) {};

\node[anchor=north west, align=left, font=\footnotesize, 
      xshift=0mm, yshift=-70mm, fill=white, draw=black] 
at (axis description cs: 0.02,0.98) {  
\begin{tabular}{@{}r@{}l@{}}
& $\Gamma_0:$ \tikz{\draw[black,dashed] (0,0) -- (0.5,0);} \\
& $\Gamma_+:$  \tikz{\fill[blue!50] (0,0) rectangle (0.5,0.2);}  \\
& $\Gamma_-:$ \tikz{\fill[red!50] (0,0) rectangle (0.5,0.2);}  
\end{tabular}};

\end{axis}
\end{tikzpicture}
\caption{Bifurcation diagram in the parameter plane $(C_1, C_2)$}  
\label{fig:1}
\end{figure}
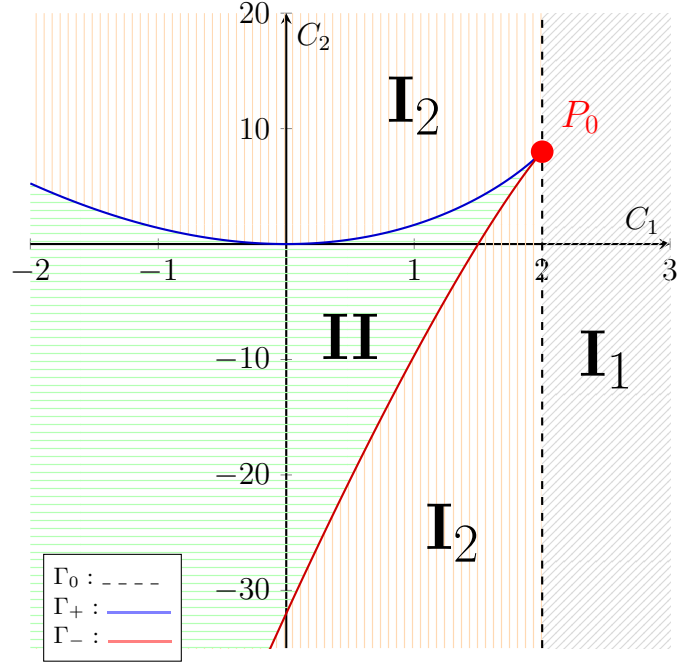

Moreover, the classification of $\kappa(x,y)=h(y/x)$ is given by the following Table \ref{tab:kap}\footnote{Here, we denote the free parameter by $C$ instead of the common integration constant $C_4$. The symbols $I,$ $P_0,$ II and $\Gamma_{\pm}$ are given by \eqref{eq:domain}.}.
We emphasize that the values of the parameters $a$, $b$, $c$, $\alpha$, and $\beta$ may vary with respect to the choice of $(C_1,C_2)$ as in the last subsection.

\begin{longtable}{@{} L{3.8cm} L{\dimexpr\textwidth-4.5cm} @{}} 
\caption{Classification of $\kappa$}\\
\toprule
\textbf{Range of $(C_1,C_2)$ } & \textbf{Expression of} $\kappa(x,y)$ \\
\midrule
\endfirsthead

\multicolumn{2}{c}{\textbf{Table \thetable\ Continued}}\\
\toprule
\textbf{Range of $(C_1,C_2)$} & \textbf{Expression of} $\kappa(x,y)$ \\
\midrule
\endhead

\bottomrule
\multicolumn{2}{r}{Continued on next page}
\endfoot

\bottomrule
\endlastfoot


\text{Singular solution} & 
$C$ \\

\text{I} & 
$\alpha - \beta \cot^2\left(\dfrac{1}{2}\,\mathrm{am}\left(
\dfrac{\sqrt{6\beta}}{3}\arctan\,\left(y/x\right) + \sqrt{\beta}C,\ 
\sqrt{\dfrac{\beta - m + \alpha}{2\beta}}\right)\right)$ \\
\addlinespace[3mm]

$P_0$ & 
$-2 - \dfrac{6}{\left( \arctan \, (y/x) + C\right)^2}$ \\
\addlinespace[3mm]

II 
& 
$\begin{cases}
c - (c - b) \sin^2\left( \mathrm{am}\left( \sqrt{\dfrac{c-a}{6}} \arctan\,\left(y/x\right) +\sqrt{c-a} \,C, \sqrt{\dfrac{c-b}{c-a}} \right) \right)\\
\quad\text{or} \\
a - (b - a) \tan^2\left( \mathrm{am}\left( \sqrt{\dfrac{c-a}{6}} \arctan\,\left(y/x\right) +\sqrt{c-a} \,C, \sqrt{\dfrac{c-b}{c-a}} \right) \right)
\end{cases}$ \\
\addlinespace[3mm]



$\Gamma_+$ 
& $-2-2\sqrt{2(2-C_1)} -3\sqrt{2(2-C_1)} \tan^2\left( \sqrt[4]{\dfrac{2-C_1}{2}}\arctan\,\left(y/x\right) + C \right)$ \\
\addlinespace[3mm]

$\Gamma_-$ 
& $\begin{cases} 
\text{\fontsize{10pt}{11pt}\selectfont $\displaystyle
-2 + 2\sqrt{2(2 - C_1)} - 3\sqrt{2(2 - C_1)} \cdot \tanh^2\left( \frac{\sqrt[4]{2 - C_1}}{2}
\left( \sqrt[4]{8} \arctan\,(y/x) + C \right) \right)
$}
\\
\quad\text{or}\\
\text{\fontsize{10pt}{11pt}\selectfont $\displaystyle
-2 + 2\sqrt{2(2 - C_1)} - 3\sqrt{2(2 - C_1)} \cdot \coth^2\left( \frac{\sqrt[4]{2 - C_1}}{2}
\left( \sqrt[4]{8} \arctan\,(y/x) + C \right) \right)
$}
\end{cases}$
\vspace{2mm}
\label{tab:kap}
\end{longtable}

\begin{remark}
For fixed $(C_1,C_2)$ in region \text{II} or on $\Gamma_-$,
there exist two distinct classes of solutions,
arising from two admissible ranges of $h$ in \eqref{eq:I_C_1}.
Indeed, the condition $\mathcal{P}_3(h)<0$ 
yields two such ranges. In region $\mathrm{II}$, where
$\mathcal{P}_3$ has three distinct real roots $a<b<c$, these ranges are
$h<a$ and $b<h<c$. On $\Gamma_-$, where
$\mathcal{P}_3(h)=(h-b)(h-a)^2$ with $a<b$, the ranges are instead
$h<a$ and $a<h<b$.
\end{remark}

\subsection{Alternative approach using reciprocal invariants}
\label{subsec:reciprocal}
 Alternatively, we can also perform a reduction of the problem \eqref{eq:0forceNS}  using a different set of invariants:
\begin{equation}\label{eq:invs_2}
    \wt z:=x/y,\quad 
    \widetilde{u}(\wt z):=yu(x,y),\quad 
    \widetilde{v}(\wt z):=yv(x,y), \quad 
    \widetilde{p}(\wt z):=y^2p(x,y)
\end{equation}
for $(x,y)$ in a cone-shaped domain $\Omega \subset \Rs$ with vertex at the origin 
and \emph{not intersecting with the $x$-axis}. 
We shall see that the invariants \eqref{eq:invs_2} will \emph{not} give us essentially different (local) self-similar solutions of \eqref{eq:0forceNS} from the ones obtained in Subsection \ref{subsec:summary}.

Similarly, the new variables $\widetilde{u}$, 
$\widetilde{v}$ and $\widetilde{p}$ satisfy the following reduced ODE system:
 \begin{equation}\label{x/y}
 \left\{
 \begin{array}{l}
 (\wt z^2 + 1)\widetilde{u}_{ \wt z \wt z} +  \wt z\widetilde{v}\widetilde{u}_{\wt z} - \widetilde{u}_{\wt z}\widetilde{u} + 4\wt z\widetilde{u}_{\wt z} + \widetilde{v}\widetilde{u} + 2\widetilde{u} - \widetilde{p}_{\wt z} =0, \vspace{0.1cm}\\
 	(\wt z^2 + 1)\widetilde{v}_{\wt z\wt z} + \wt z\widetilde{v}\widetilde{v}_{\wt z} - \widetilde{v}_{\wt z}\widetilde{u} + 4\wt z\widetilde{v}_{\wt z} + \widetilde{v}^2 + 2\widetilde{v} + \wt z\widetilde{p}_{\wt z} + 2
 	\widetilde{p} = 0, \vspace{0.1cm}\\
 	\wt z\widetilde{v}_{\wt z} + \widetilde{v} - \widetilde{u}_{\wt z} = 0.
 \end{array}
 \right.
 \end{equation}
 
Assuming the same condition \eqref{cdt:xnot0}, the last equation in \eqref{x/y} gives
 \begin{equation}\label{eq:u=v}
 \widetilde{u} = \wt z\,\widetilde{v}.
 \end{equation}
Then substituting \eqref{eq:u=v} into \eqref{x/y} yields
\begin{subequations}
\begin{empheq}[left=\empheqlbrace]{align}
  & (\wt z^3+\wt z)\widetilde{v}_{\wt z\wt z} + (6\wt z^2+2)\widetilde{v}_z 
 + \wt z\widetilde{v}^2 + 6\wt z\widetilde{v} - \widetilde{p}_{\wt z} =0,\label{eq:zva_a}  \vspace{0.1cm} \\ 
&  (\wt z^2+1)\widetilde{v}_{\wt z\wt z} + 4\wt z\widetilde{v}_{\wt z} + \widetilde{v}^2 
  + 2\widetilde{v} + \wt z\widetilde{p}_{\wt z} + 2\widetilde{p} =0.\label{eq:zva_b}
\end{empheq}
\end{subequations}

Now, multiplying both sides of \eqref{eq:zva_a} by $\wt z$  and summing up the resulting equation and \eqref{eq:zva_b}, we obtain that
\begin{equation}\label{eq:zvb+a}
   \begin{aligned}
     \widetilde{p}=-\frac{1}{2}\left( (\wt z^2+1)^2\widetilde{v}_{\wt z\wt z} + 6\wt z(\wt z^2+1)\widetilde{v}_{\wt z} 
      +(\wt z^2+1)\widetilde{v}^2 + (6\wt z^2+2)\widetilde{v}\right).
 \end{aligned} 
 \end{equation}
 On the other hand, by multiplying both sides of \eqref{eq:zva_b} by $\wt z$ and subtracting \eqref{eq:zva_a} from the resulting equation, we have
 \begin{equation*}
 \frac{d}{d\wt z} \left( (\wt z^2+1)\widetilde{p}-2(\wt z^2+1)\widetilde{v}\right)=0,
 \end{equation*}
 which implies
\begin{equation}\label{eq:zva-b}
     \widetilde{p}=2\widetilde{v} +\frac{\wt C_1}{\wt z^2+1}
 \end{equation}
 for some real number $\wt C_1$.
 Equating \eqref{eq:zvb+a} and \eqref{eq:zva-b}, we get
\begin{equation}\label{eq:zva-c}
    (\wt z^2+1)^2\widetilde{v}_{\wt z\wt z} + 6\wt z(\wt z^2+1)\widetilde{v}_{\wt z} 
 + (\wt z^2+1)\widetilde{v}^2 + 6(\wt z^2+1)\widetilde{v} 
 +\frac{2\wt C_1}{\wt z^2+1}=0.
 \end{equation}
 Analogous to the solvability of \eqref{eq:tilu}, we define
 \begin{equation*}
   \widehat{h} := (\wt z^2+1) \widetilde{v},
 \end{equation*}
from which we can transfer \eqref{eq:zva-c} into the equation of $\widehat{h}$ as follows:
\begin{equation}\label{h_1 v}
 (\wt z^2+1)^2\widehat{h}_{\wt z\wt z}+2\wt z(\wt z^2+1)\widehat{h}_{\wt z}+\widehat{h}^2+4\widehat{h}+2\wt C_1=0.
 \end{equation}
 
Note that the forms of \eqref{h_1} and \eqref{h_1 v} are the same. Consequently, the calculations in Subsection \ref{subsec:solve} are still valid for the problem \eqref{h_1 v}.
Moreover, we obtain 
\begin{equation}\label{eq:hat_h}
\widehat{h}(\widetilde{z}) = \kappa(y, x),  \quad
 u(x,y) = \displaystyle\frac{x}{x^2+y^2}\kappa(y,x),\quad 
 v(x,y) = \displaystyle\frac{y}{x^2+y^2}\kappa(y,x),\quad 
\end{equation}
where $\kappa$ is given by Table \ref{tab:kap}. Noticing the identity 
\begin{equation*}
 \arctan \,(x/y) + \arctan\, (y/x) = \hbox{sgn}\,(x/y)\frac{\pi}{2},
\end{equation*}
the solutions \eqref{eq:hat_h} are exactly of the same form as those obtained in Subsection \ref{subsec:summary} up to the choices of free constants so long as the cone region $\Omega$ does not intersect with the $x$-axis and the $y$-axis simultaneously.

\section{Radial self-similar solutions in large cones}
\label{sec:proof}

In this section, we will complete the proof of Theorem \ref{thm:main-2}.
Roughly speaking, the last section proved that the local profiles of the radial self-similar solutions in small cones are determined by the choice of $\kappa$ in Table \ref{tab:kap}. 
For convenience, let us denote the functions in Table \ref{tab:kap} as follows:
\begin{equation}\label{eq:kappas}
    \kappa_j (x,y)=\left(f_j \circ  \theta\right) (x,y),\quad j=0,1, \dots,7,
\end{equation}
where  $\theta(x,y) = \arctan \, (y/x)$ and 
\begin{equation}\label{eq:fs}
 \begin{aligned}
 f_0(\theta)&=C,\\
f_1(\theta) & = \alpha - \beta \cot^2 \left( \dfrac{1}{2} \operatorname{am} \left( \dfrac{\sqrt{6\beta}}{3} 
\theta + \sqrt{\beta} \, C,  \sqrt{\dfrac{\beta - m + \alpha}{2\beta}}\,\, \right) \right), \\
f_2(\theta) & = -2 - \dfrac{6}{( \theta + C)^2}, \\
f_3(\theta) & = c - (c - b) \sin^2 \left( \operatorname{am} \left( \sqrt{\dfrac{c-a}{6}}  \theta + \sqrt{c-a} \, C,  \sqrt{\dfrac{c-b}{c-a}} \,\right) \right), \\
f_4(\theta) & = a - (b - a) \tan^2 \left( \operatorname{am} \left( \sqrt{\dfrac{c-a}{6}} \theta + \sqrt{c-a} \, C,  \sqrt{\dfrac{c-b}{c-a}} \,\right) \right), \\
f_5(\theta) & = -2 - \sqrt{2(2-C_1)} - 3\sqrt{2(2-C_1)} \tan^2 \left( \sqrt[4]{\dfrac{2-C_1}{2}} \theta + C \right), \\
f_6(\theta) & = -2 + 2\sqrt{2(2 - C_1)} - 3\sqrt{2(2 - C_1)}  \tanh^2\left( \frac{\sqrt[4]{2 - C_1} }{2} \left( \sqrt[4]{8} \theta + C \right) \right), \\
f_7(\theta) & =  -2 + 2\sqrt{2(2 - C_1)} - 3\sqrt{2(2 - C_1)}  \coth^2\left( \frac{\sqrt[4]{2 - C_1} }{2}  \left( \sqrt[4]{8} \theta + C \right) \right).
\end{aligned}   
\end{equation}
It should be noted that the domain of $(x,y)$ for each function $\kappa_j$ may differ and, in some cases, depends on the parameters.

\subsection{Construction of the radial flows in large cones}
In this subsection, we will study the local profiles of radial flows in some cone $\Omega$ with its angle $\alpha (\Omega) \in (\pi,2\pi].$ In particular, whenever $\alpha(\Omega)=2\pi,$ $\Omega$ should be understood as the subset $\Rs \backslash R_{\theta}$ of $\Rs$ for
$$R_{\theta} := \{(r\cos \theta ,r\sin \theta) \mid  r>0 \}$$
and for some $\theta\in [0,2\pi).$ In addition, let us denote the four quadrants of the plane as 
\begin{equation*}
        \begin{aligned}
         Q_{I} &:=\{(\wt x, \wt y)\in \Rs : \wt x>0, \,\wt y >0\}, \\ 
         Q_{II} &:=\{(\wt x, \wt y)\in \Rs : \wt x<0,\,\wt y >0\}, \\ 
          Q_{III}&:=\{(\wt x, \wt y)\in \Rs : \wt x<0, \,\wt y <0\},\\
          Q_{IV}&:=\{(\wt x, \wt y)\in \Rs : \wt x>0, \,\wt y <0\}.
        \end{aligned}
    \end{equation*}
    
According to the forms in \eqref{eq:kappas} and \eqref{eq:fs}, the possible singularities of $\kappa_j$ in \eqref{eq:kappas} may appear on the ray $R_{\theta}$ for some $\theta\in [0,2\pi).$
\begin{dfn}
Let $R_{\theta_0}$ be the ray with an angle $\theta_{0} \in [0,2\pi)$. 
Let $\Omega = \Omega_{+} \cup R_{\theta_0} \cup \Omega_{-}$ be an open simply connected cone region in $\Rs$ with its vertex at the origin such that the subsets satisfy $\Omega_{\pm} \neq \varnothing$ and $\Omega_{+} \cap \Omega_{-} = \varnothing$ (see Figure \ref{fig:ray}).  Suppose that the functions $\kappa_j \in C^{\infty}(\Omega_+\cup \Omega_-)$ ($1\leq j \leq 7$). For any $(x_0,y_0)\in R_{\theta_0}$, let us denote the non-tangential limits (if they exist) as
\begin{equation*}
      L_{\pm}^{\alpha}(x_0,y_0;\kappa_j) :=\lim_{\delta\to 0^{+}} 
      (\pd^{\alpha}\kappa_j) \big|_{(x_0,y_0) \pm \delta \bn (x_0,y_0)},
\end{equation*}
where $\bn$ is the unit normal vector on $R_{\theta_0}$ pointing from $\Omega_-$ to $\Omega_+$ and $\alpha\in \BN_0^2$.

\begin{figure}[H]
\centering
\begin{tikzpicture}[scale=1.8, >=Stealth]
    \draw[->, black] (-0.5,0) -- (3,0) node[right] {$x$};
    \draw[->, black] (0,-0.5) -- (0,3) node[above] {$y$};
    
    \fill[green!15, opacity=0.6] (0,0) -- (25:2.8) arc (25:65:2.8) -- cycle;
    \fill[red!15, opacity=0.6] (0,0) -- (65:2.8) arc (65:105:2.8) -- cycle;
    
    \draw[black] (0,0) -- (65:2.8) coordinate (R) node[above right] {$R_{\theta_0}$};
    
    \node at (45:1.3) [green!60!black] {$\Omega_+$};
    \node at (81:1.3) [red!70!black] {$\Omega_-$};
    
    \coordinate (P) at (65:2.0);
    \fill (P) circle (1.3pt) node[above left=0.5pt] {$(x_0,y_0)$};

    \draw[->, thick, black] (P) -- ++(-25:0.5) node[above right=0.5pt] {$\bn$};

    \draw[gray!40, thin] (25:2.8) arc (25:105:2.8);
    \draw[gray!40, thin, dashed] (0,0) -- (25:2.8);
    \draw[gray!40, thin, dashed] (0,0) -- (105:2.8);
    
\end{tikzpicture}
\caption{The cone domain $\Omega = \Omega_{+} \cup R_{\theta_0} \cup \Omega_{-}$}
\label{fig:ray}
\end{figure}
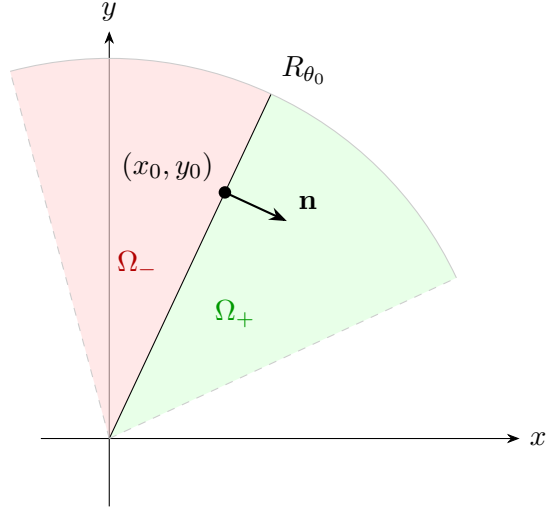

\begin{itemize}
    \item Assume that $L_{\pm}^{\alpha}(x_0,y_0;\kappa_j)$ exist for any multi-index $\alpha\in \BN_0^2$. If, in addition, there exists some  $\alpha_0\in \BN_0^2$ such that
     \begin{equation*}
L_{+}^{\alpha_0}(x_0,y_0;\kappa_j)
      \neq  L_{-}^{\alpha_0}(x_0,y_0;\kappa_j),
    \end{equation*}
    then we define $(x_0,y_0)$ as a \emph{weak} singular point of $\kappa_j$.
    
     \item If there exists some $\alpha_0\in \BN_0^2$ such that 
     \begin{equation*}
      \big|L_{+}^{\alpha_0}(x_0,y_0;\kappa_j) \big| \,\,\text{or}\,\,
       \big|L_{-}^{\alpha_0}(x_0,y_0;\kappa_j) \big| = \infty,   
      \end{equation*}
then we define $(x_0,y_0)$ as a \emph{strong} singular point of $\kappa_j$.
    \end{itemize}
\end{dfn}

\begin{remark}\label{rmk:ssp}
Consider the functions $f_j(\theta)$ defined in \eqref{eq:fs}. 
If 
\begin{equation*}
  \Big|\lim_{\theta \to \theta_0^{+}} f_j(\theta)\Big|  
  \,\,\,\text{or}\,\,\, 
  \Big| \lim_{\theta \to \theta_0^{-}} f_j(\theta)\Big|= \infty 
\end{equation*}
for some point $(x_0,y_0)$ with $ \arctan\,(y_0/x_0) = \theta_0$, then $ \kappa_j$ exhibits a \emph{strong} singularity at $(x_0,y_0)$.
In such a situation, we cannot find any smooth extension of $\kappa_j$ beyond the ray $R_{\theta_0}$ when $\kappa_j$ is well-defined for $\theta>\theta_0$ or $\theta <\theta_0$.
\end{remark}

To prove Theorem \ref{thm:main-2}, we assume, without loss of generality, that 
\begin{equation}\label{eq:Omega}
\alpha(\Omega) \in (3\pi/2,2\pi],\quad 
R_0\nsubseteq \Omega, \quad 
\Omega \cap Q_{I},\,\,\Omega \cap Q_{IV}\not=\emptyset
\quad \text{and}\quad  
Q_{II},Q_{III}\subset \Omega.
\end{equation}
In addition, we suppose that there exists $\kappa\in C^{\infty}(\Omega)$ such that the problem \eqref{eq:NS_0} admits some smooth radial self-similar solution in $\Omega:$
\begin{equation}\label{eq:s-sol}
u=\frac{x}{x^2+y^2} \kappa \quad \text{and} \quad  v=\frac{y}{x^2+y^2} \kappa.
\end{equation}
Furthermore, we can prove such $\kappa$ in \eqref{eq:s-sol} is also real analytic in $\Omega$, i.e., $\kappa\in C^{\omega}(\Omega)$. In fact, assume that $(x_0,y_0)$ is any given point in $\Omega$ and $\Omega'$ is some (small enough) cone neighborhood of $(x_0,y_0)$ with its vertex at the origin. 
By the discussion in Section \ref{sec:symm}, we have 
\begin{equation*}
    \kappa (x,y)= f_j \big( \arctan \, (y/x) \big)
    \,\,\, \text{or} \,\,\,f_j \big( \arctan \, (x/y) \big)
\end{equation*}
for some $j\in \{0,1,2,\dots,7\}$ and for any $(x,y)\in \Omega'$. 
Thus, $\kappa \in C^{\omega}(\Omega')$ due to the formula of $f_j$ in \eqref{eq:fs}.
\medskip

Inspired by the calculations in Section \ref{sec:symm}, it is reasonable to assume that the restriction of $\kappa$ to the first quadrant coincides with some $\kappa_{j}$ $(j=0,1,\dots,7)$ given in \eqref{eq:kappas}. Certainly, we also have to choose suitable parameters such that  $\kappa_{j} \in C^{\infty}(\Omega\cap Q_{I})$. Moreover, we can prove the following result.
\begin{prop}\label{prop:ext}
Assume that the domain $\Omega$ is a cone satisfying \eqref{eq:Omega}.
Let $f_j$ be defined in \eqref{eq:fs} for any $j\in \{1,2,\dots,7\}$, and 
let $\kappa \in C^{\infty}(\Omega)$ be given as in \eqref{eq:s-sol}.
Let us introduce the auxiliary function 
\begin{equation}\label{wt theta}
    \wt \theta (x,y)=\begin{cases}
     \arctan\,(y/x)
    & \text{if}\,\, x>0,y\geq 0,\\
       \pi/2-\arctan\, (x/y)
    & \text{if}\,\, x\leq 0,y> 0,\\
 \pi+\arctan\,(y/x)
    & \text{if}\,\, x<0,y\leq 0,\\
 3\pi/2-\arctan\,(x/y)
  & \text{if}\,\, x\geq 0,y< 0,
    \end{cases}
\end{equation}
where the range of $\arctan\, z$ belongs to $[-\pi/2,\pi/2]$. If the restriction of $\kappa$ to $\Omega\cap Q_{I}$ coincides with some $\kappa_j \in C^{\infty}(\Omega\cap Q_{I})$ in \eqref{eq:kappas}, and if we denote $\wt{\kappa}_j := f_j \circ \wt{\theta}$, then $\kappa = \wt{\kappa}_j$ in $\Omega$.
\end{prop}

\begin{proof}
By the assumption $\kappa=\wt\kappa_j$ in $\Omega\cap Q_I$, we can regard $\wt\kappa_j$ as a function in the class $C^0(\overline{\Omega\cap  Q_I})$. In particular, we have
\begin{equation*}
\Big| \lim_{\theta \to (\pi/2)^-} f_j(\theta)  \Big| <\infty  .  
\end{equation*}
On the other hand, using the definition of $f_j$, i.e., \eqref{eq:fs}, we know that $f_j$ is smooth near the point $\theta=\pi/2$. Therefore, there exists a small cone $\Omega'$ containing the ray $R_{\pi/2}$ such that 
$\wt \kappa_j \in C^{\infty}(\Omega' \cup Q_I)$. 
Furthermore, $\wt \kappa_j \in C^{\omega}(\Omega' \cup (\Omega\cap  Q_I))$ by the definitions of $f_j$ and $\wt \theta$. As both $\kappa$ and $\wt \kappa_j$ can be regarded as the real analytic extension of $\kappa_j$ into 
$\Omega' \cup (\Omega\cap  Q_I)$, we use the property of unique continuation of real analytic functions 
(see, e.g., \cite[Section 3 (b) of Chapter 3]{John1982}), to conclude 
\begin{equation}\label{eq:wt_kappa_1}
    \kappa = \wt \kappa_j \quad \text{in} \,\,\, \Omega' \cup (\Omega\cap  Q_I).
\end{equation}

Next, we show that $\kappa$ and $\wt \kappa_j$ coincide in the region $\Omega_+:= (\Omega\cap  Q_I)\cup R_{\pi/2} \cup Q_{II}$, that is,
\begin{equation}\label{eq:wt_kappa_2}
    \kappa = \wt \kappa_j \quad \text{in} \,\,\, \Omega_+.
\end{equation}
Suppose that \eqref{eq:wt_kappa_2} does not hold and $\Omega'$ is the maximal (open) cone region in $\Omega_+$ satisfying \eqref{eq:wt_kappa_1}. Then there exists $(x_0,y_0) \in \pd \Omega' \cap Q_{II}$ such that $\wt\kappa_j$ admits the strong singularity at $(x_0,y_0)$. This contradicts \eqref{eq:wt_kappa_1} and the choice of $\kappa$ in the class $C^{\infty}(\Omega)$. 
Consequently, $\wt \kappa_j$ has no singular point in $\Omega_+$ and \eqref{eq:wt_kappa_2} holds true.

In a similar manner, we can prove that $\kappa$ and $\wt \kappa_j$ have the same values as well whenever $(x,y)$ belongs to $R_{\pi}$, $Q_{III}$, $R_{3\pi/2}$ and $\Omega\cap Q_{IV}$. This completes the proof of $\kappa = \wt \kappa_j$ in $\Omega.$
\end{proof}

\begin{remark}
A straightforward computation shows that the auxiliary function $\widetilde{\theta}$ in \eqref{wt theta} is continuous on $\Rs\setminus R_0$, and that all its partial derivatives are smooth throughout $\Rs$. Indeed, we have
\begin{equation}\label{partial theta smooth}
\begin{aligned}
   \pd_{x,y}^{\alpha} \big( \wt\theta(x,y) \big) 
   = \pd_{x,y}^{\alpha}\big( \arctan \,(y/x) \big)
   =  \frac{P_{|\alpha|}(x,y)}{(x^2+y^2)^{|\alpha|}}\in C^{\infty} (\Rs)
\end{aligned}
\end{equation}
for any multi-index $\alpha\in \BN_0^2$ with $|\alpha| \geq 1$ and for some polynomial $P_{|\alpha|}$ of order $|\alpha|$.
In addition, $P_{|\alpha|}$ is a linear combination of homogeneous terms $x^{\beta_1}y^{\beta_2}$ with $\beta_1+\beta_2=|\alpha|$.
\end{remark}

At last, the following corollary directly from Proposition \ref{prop:ext}.
\begin{cor}\label{cor:ext_Rs}
Let $f_j$ be defined in \eqref{eq:fs} for $j\in \{1,2,\dots,7\}$, and let $\wt \theta (x,y)$ be defined in \eqref{wt theta}. Suppose that $\wt \kappa=f_j \circ \wt{\theta}\in C^{\infty}(\Rs\backslash R_0)$ for some $j$ has been constructed as in Proposition \ref{prop:ext} whenever the cone $\Omega=\Rs\backslash R_0$. If there exists some $\wt\kappa\in C^{\infty}(\Rs)$ such that
\begin{equation*}
 \wt \kappa|_{\Rs\backslash R_0}=\wt \kappa_j,   
\end{equation*}
then the $n$-th order derivative $f_{j}^{(n)}$ of $f_j$ with respect to $\theta$ necessarily satisfies
\begin{equation}\label{cdt:f_j}
   f^{(n)}_j(0)=f_{j}^{(n)}(2\pi)
\end{equation}
for any nonnegative integer $n$.
\end{cor}

\subsection{From the local profiles to radial Jeffery--Hamel solutions}
\label{subsec:profiletoJH}
In this subsection, we will discuss the relationship between the local profiles obtained in $\Rs\backslash R_0$ and the radial Jeffery--Hamel solutions in Definition \ref{def:JH}.  To this end, we suppose that
\begin{equation}\label{eq:s-sol-2}
u=\frac{x}{x^2+y^2} \kappa \quad \text{and} \quad  v=\frac{y}{x^2+y^2} \kappa.
\end{equation}
is a smooth radial self-similar solution of \eqref{eq:NS_0} for some $\kappa \in C^{\infty}(\Rs).$ 
More precisely, we shall see that such $\kappa$ in \eqref{eq:s-sol-2} can only be regarded as the extension from $\kappa_0$ or $\kappa_3,$ which is not very surprising based on the results in \cite{sverak2011landau,guillod2015generalized}. Obviously, if we take $\kappa=C$ in \eqref{eq:s-sol-2} for some constant $C,$ then \eqref{eq:s-sol-2} provides us with a smooth radial self-similar solution of \eqref{eq:NS_0}. Thus, the main task of this subsection is to study $\kappa_j$ for $1\leq j\leq 7.$ The key point in the following analysis is to use the necessary conditions in \eqref{cdt:f_j} and the distribution of singularities of $\kappa_j$ in \eqref{eq:kappas}.
\smallbreak 

In view of Proposition \ref{prop:ext}, we only need to focus on the restriction of $\kappa$ in the first quadrant $Q_I$ as $\eqref{eq:s-sol-2}$ can be regarded as a local solution in $Q_I.$

\begin{lem}
Let $\kappa \in C^{\infty}(\Rs)$ be given as in \eqref{eq:s-sol-2}.  
Then the restriction of $\kappa$ to $Q_{I}$ cannot coincide with $\kappa_2$, $\kappa_6$ or $\kappa_7$
in \eqref{eq:kappas}.
\end{lem}

\begin{proof}
Suppose that $\kappa|_{Q_{I}}=\kappa_2 \in C^{\infty}(Q_I)$. 
According to Corollary \ref{cor:ext_Rs}, we have $f_2^{(n)}(0) = f_2^{(n)}(2\pi)$ for $n=0,1$.
By the condition $f_2(0) = f_2(2\pi)$, there holds 
\[\frac{1}{C^2} = \frac{1}{(2\pi + C)^2} .\]
Then we have to choose $C = -\pi$ in the formula of $f_2(\theta)$ (see \eqref{eq:fs}). 
However, such a choice of $C$ leads to  
$$ f_2'(\theta) = \frac{12}{(\theta -\pi)^3},$$
from which we easily see $f_2'(0) \neq f_2'(2\pi)$, which is a contradiction and hence $\kappa|_{Q_{I}} \neq  \kappa_2$.  

As the discussions on the cases $\kappa_6$ and $\kappa_7$ are similar, we only give the details concerning $\kappa_6$. Assuming $\kappa|_{Q_{I}}=\kappa_6 \in C^{\infty}(Q_I)$, 
Corollary \ref{cor:ext_Rs} provides us with  $f_6^{(n)}(0) = f_6^{(n)}(2\pi)$ for $n=0,1$.
For simplicity, let us denote 
\begin{equation*}
f_6(\theta) = -2 + 2B_6 - 3B_6 \tanh^2 \big( g_6(\theta) \big),     
\end{equation*}
where we have defined that\footnote{Note that $C_1<2$ whenever $(C_1,C_2)\in \Gamma_-$.}
$$ B_6= \sqrt{2(2 - C_1)}>0, \quad g_6(\theta) = A_6 \left( \sqrt[4]{8} \theta + C \right)  \quad \text{and}\quad 
A_6 = \frac{(2 - C_1)^{1/4}}{2}>0. $$
Now, the condition $f_6(0) = f_6(2\pi)$ is reduced to the equality 
\begin{equation*}
 \tanh^2\big( g_6(0) \big) = \tanh^2\big( g_6(2\pi) \big),  
\end{equation*}
which yields $ |g_6(0)| = |g_6(2\pi)|$. Thus, we have  
\begin{equation}\label{eq:CA_6}
    C = -\sqrt[4]{8} \pi, \quad 
    g_6 (0)=-\sqrt[4]{8} \pi A_6  \quad \text{and}\quad 
g_6(2\pi) =  \sqrt[4]{8} \pi A_6.
\end{equation}
On the other hand, inserting \eqref{eq:CA_6} into the following formula 
$$ f_6'(\theta) = -6\sqrt[4]{8} A_6 B_6 \cdot 
\tanh \big( g_6(\theta)\big) \cdot \text{sech}^2\big( g_6(\theta)\big),$$
we have
$$f_6'(0) = -f_6'(2\pi)=6\sqrt[4]{8} A_6 B_6 \cdot  
\tanh\left(  \sqrt[4]{8} \pi A_6 \right) \cdot 
\text{sech}^2 \left(  \sqrt[4]{8} \pi A_6 \right) \neq 0,$$
which contradicts the requirement $f_6'(0) = f_6'(2\pi)$.
Hence $\kappa|_{Q_{I}} \neq  \kappa_6$.
\end{proof}

Next, we will prove that $\kappa$ cannot coincide with $\kappa_1$, $\kappa_4$ or $\kappa_5$ in $Q_{I}$ due to the distribution of singular points.

\begin{lem}
Let $\kappa \in C^{\infty}(\Rs)$ be given as in \eqref{eq:s-sol-2}.  
Then the restriction of $\kappa$ to $Q_{I}$ cannot coincide with $\kappa_1$, $\kappa_4$ or $\kappa_5$ in \eqref{eq:kappas}.
\end{lem}

\begin{proof}
Suppose that $\kappa|_{Q_{I}}=\kappa_1 \in C^{\infty}(Q_I)$. Let us denote 
    \begin{equation*}
        f_1(\theta)  = \alpha - \beta \cot^2 \left( g_1(\theta) \right)
    \end{equation*}
with 
\begin{equation*}
g_1(\theta):= \dfrac{1}{2} \operatorname{am}\, \left(A_1 
\theta + \sqrt{\beta} \, C,  k\, \right),\quad 
A_1=\dfrac{\sqrt{6\beta}}{3} >0 
\quad \text{and}\quad  
k=\sqrt{\dfrac{\beta - m + \alpha}{2\beta}}.  
\end{equation*}
By Corollary \ref{cor:ext_Rs}, we have $\kappa(x,y) = \big(f_1 \circ \wt \theta\,\big) (x,y)$ in $\Rs$ with $\wt \theta$ defined by \eqref{wt theta}.
Moreover, we have $f_1^{(n)}(0) = f_1^{(n)}(2\pi)$ for $n=0,1$ in this situation.
From the expression of $ f_1 $, the condition $ f_1(0) = f_1(2\pi) $ is admissible only if
\begin{equation*}
\cot \big( g_1(0) \big) = \pm \cot \big( g_1(2\pi) \big).
\end{equation*}
\begin{itemize}
    \item Suppose that $\cot \big( g_1(0) \big) = -\cot \big( g_1(2\pi) \big) \neq 0$.
The condition $f_1'(0) = f_1'(2\pi)$ leads to the following equality:
\begin{equation}\label{eq:k_1_1}
-\frac{\operatorname{dn} \left( \sqrt{\beta}\,C, k \right) }{\sin^2 \left( g_1(0) \right)} = \frac{\operatorname{dn} \left( 2\pi A_1 + \sqrt{\beta}\,C, k \right) }{\sin^2 \left( g_1(2\pi)\right)},    
\end{equation}
where we use the notation 
\begin{equation*}
\operatorname{dn}\,(\mu,k_0):=\frac{\pd \operatorname{am}\,(\mu,k_0)}{\pd \mu} 
\end{equation*}
for any fixed $k_0$.
However, \eqref{eq:k_1_1} can never be true as $\operatorname{dn} > 0$.

\item Suppose that $\cot \big( g_1(0) \big) = \cot \big( g_1(2\pi) \big)$. Note that
 $g_1(\theta)$ is continuous and strictly increasing. Hence we have $g_1(2\pi)-g_1(0)\geq \pi $ as the cotangent function has a period $\pi$. Therefore, there exists some $\theta_0 \in (0,2\pi)$ such that  
 \begin{equation*}
     \lim_{\theta \to \theta_0} \big|\cot\big(g_1(\theta)\big)\big|=\infty.
 \end{equation*}
By the relation above, $\kappa_1$ admits some strong singularity at the ray $R_{\theta_0}$.
Then we have $\kappa|_{Q_{I}} \neq \kappa_1$ due to Remark \ref{rmk:ssp}
\end{itemize}
The discussions for $\kappa|_{Q_I} = \kappa_j \in C^{\infty}(Q_I)$, $j = 4,5$, are analogous, and we omit the details here.

\end{proof}

The following proposition states that the choice of $\kappa_3$ provides us with a smooth self-similar solution of \eqref{eq:NS_0}.
\begin{prop}
\label{prop:exist}
Assume that $C\in \mathbb{R}$ and $(C_1,C_2)\in \text{II}$ as in \eqref{eq:domain}. Let $\wt \theta(x,y)$ be given by \eqref{wt theta}. For any $(x,y) \in \Rs$, let us define 
\begin{equation*}
\begin{aligned}
         \omega_c (x,y) & :=
          \cos \left( 2\,\mathrm{am}\left(
                   \sqrt{\dfrac{c-a}{6}}   \wt \theta (x,y)+\sqrt{c-a}C , \sqrt{\dfrac{c-b}{c-a}} 
                   \, \right) \right),\\
             \wt \kappa_3 (x,y) &:= c-(c-b) \frac{1-\omega_c(x,y)}{2},
\end{aligned}
\end{equation*}
where $a < b < c$ are real roots of the cubic equation $h^3 + 6h^2 + 6C_1 h + C_2 = 0$ satisfying the relation  
\begin{equation}\label{cdt:K}
\pi \sqrt{\dfrac{c-a}{6}}  = n K \left( \sqrt{\frac{c - b}{c - a}} \,\right)
\end{equation}
for some positive integer $n$. Here, $ K(k) =F(\pi/2,k)$ (see \eqref{eq:am} for the definition of the complete elliptic integral of the first kind).
Then the functions 
\begin{equation*}
\begin{aligned}
u(x,y) &= \frac{x}{x^2+y^2}\,\wt\kappa_3(x,y), \quad
v(x,y) = \frac{y}{x^2+y^2}\,\wt\kappa_3(x,y),\\
p(x,y) &= \frac{2}{x^2+y^2}\,\wt\kappa_3(x,y) + \frac{C_1}{x^2+y^2},
\end{aligned}
\end{equation*}
are smooth in $\Rs$ and satisfy the system \eqref{eq:NS_0}.
\end{prop}

In fact, Proposition \ref{prop:exist} immediately follows from Theorem \ref{thm:main-1} (see also Section \ref{sec:symm}) and the following technical lemma whose proof is given in Appendix \ref{prf lem:cos}. 
\begin{lem} \label{lem:cos}
Under the assumptions of Proposition \ref{prop:exist}, we have $\omega_c \in C^{\infty}(\Rs)$.
\end{lem}

Let us conclude this subsection with a few remarks on Proposition \ref{prop:exist}.
\begin{itemize}
    \item One can easily verify the equality $\wt{\kappa}_3 =f_3 \circ \widetilde{\theta}$ for $f_3$ in \eqref{eq:fs}. 

    \item The equivalent form of the condition \eqref{cdt:K} has already been observed in \cite[Section 2.2]{guillod2015generalized},
    where the existence of the parameters $a,b,c$ depends on some flux condition.

\end{itemize}


\section{Local profiles of non-radial self-similar solutions}\label{sec:nonradial}

In this section, we consider non-radial self-similar solutions of the SNS equations \eqref{eq:NS_0}, i.e., \eqref{eq:0forceNS}, in a cone-shaped domain $\Omega \subset \Rs$ with vertex at the origin and \emph{not intersecting with the $y$-axis}. We first show that some non-radial self-similar solutions of the SNS equations can be obtained by reducing the problem to a Li\'enard equation, and then provide some explicit solutions expressed in terms of Weierstrass elliptic functions.

Setting $\mathcal{F}(z)=0$ in \eqref{eq:zua-c} and taking
$C_0\neq0$, we obtain
\begin{equation}\label{eq:C_0}
   (z^2+1)^2\widetilde{u}_{zz} + (6z-C_0)(z^2+1)\widetilde{u}_{z} 
+ (z^2+1)\widetilde{u}^2 + 6(z^2+1)\widetilde{u} 
+\frac{2C_0(z^3+3z)}{z^2+1}
+\frac{2C_1}{z^2+1}=0.
\end{equation}
Here $\widetilde{u}$ is defined by \eqref{eq:invs} and $C_0\neq 0$ is the real constant in \eqref{C_0-neq-0}.
Substituting $h := (z^2+1) \widetilde{u}$ in \eqref{eq:C_0}, we get 
 \begin{equation}\label{eq:h_C0}
  (z^{2}+1)^{2} h_{zz} 
  + (2z - C_{0})(z^{2} +1) h_z
  + h^{2} 
  + (2C_{0}z+4)h
  + 2C_{0}z^{3}+6C_{0}z+2C_1
  =0.
\end{equation}

Firstly, assume that
\begin{equation}\label{eq:con}
C_0^2 \leq 4-2C_1 .
\end{equation}
Under this condition, Eq. \eqref{eq:h_C0} admits the following linear solution
\begin{equation}\label{eq:hL}
    h_{L}(z)=-C_0 z+\wt C_1,
\end{equation}
where we have set
\begin{equation}\label{eq:td_C1}
  \wt C_1:= -2+\sqrt{4-C_0^2-2C_1}
\quad \text{or} \quad
-2-\sqrt{4-C_0^2-2C_1}.  
\end{equation}
The linear solution \eqref{eq:hL} corresponds to the self-similar solution \eqref{sol:CC} (see also \eqref{eq:U_L})
\begin{equation}\label{eq:U_L1}
\begin{aligned}
u_{L}(x,y) := \frac{\wt C_1 x-C_0 y}{x^2+y^2}, \quad
v_{L}(x,y) := \frac{ \wt C_1 y +C_0 x}{x^2+y^2},\quad 
p_{L}(x,y) :=  - \frac{ \wt C_1^2+C_0^2}{2(x^2+y^2)}.
\end{aligned}
\end{equation}
Let us emphasize that $\wt C_1$ in \eqref{eq:U_L1} can be any real number due to \eqref{eq:td_C1}.
\medskip

To obtain more general non-radial self-similar solutions, by defining the difference 
\begin{equation} \label{eq:hhh1}
    h_r(z):=h(z)-h_L(z)
\end{equation}
and substituting it into  \eqref{eq:h_C0}, 
we obtain the equation for $h_r(z)$ as follows
\begin{equation}\label{eq:hr_pm1}
(z^{2}+1)^{2} \frac{d^2h_r}{dz^2} +(2z-C_{0})(z^{2}+1) \frac{d h_r}{dz }+h_{r}^{2} +2(\wt C_1+2)h_{r}=0.
\end{equation}
The only possible linear solution of \eqref{eq:hr_pm1} is $h_r(z)\equiv -2(\wt C_1+2)$
for a fixed real number $\wt C_1$. In this case, no new solution $h(z)$ is obtained.
Therefore, without loss of generality, we assume henceforth that $h_r(z)$ is a nonconstant solution of \eqref{eq:hr_pm1}.

\smallbreak 

Next, let us denote 
$$H(\theta):=h_{r}(\tan \theta)$$ 
for $\theta:=\arctan\,z$. Then, \eqref{eq:hr_pm1} becomes  \eqref{eq:H_pm}, i.e.,
\begin{equation}\label{eq:Heq}
\frac{d^2 H}{d\theta^2}-C_{0}\frac{d H}{d\theta} + H^{2} + 2( \wt C_1 +2) \,H=0,
\end{equation}
which is a Li\'enard equation (see, e.g., \cite{strogatz2024nonlinear}). The integrability of Li\'enard-type equations has been well studied; in particular, it is integrable 
when 
\begin{equation}\label{eq:CC1}
  -\frac{ 2(\wt C_1 +2)}{C_0^2} = \pm \frac{6}{25},\quad \text{i.e.,}\quad  \wt C_1= -2\mp \frac{3C_0^2}{25}.
\end{equation}
As a consequence of \eqref{eq:td_C1}, we have
$$C_1=2- \frac{C_{0}^{2}}{2} - \frac{9 C_{0}^{4}}{1250}.$$
The integrable Li\'enard equation \eqref{eq:Heq} under condition \eqref{eq:CC1}, has been classified as type-II integrable ODEs in the Painlev\'e--Gambier classification. It is known that the Painlev\'e--Gambier type-II equations admit solutions expressible in terms of Weierstrass elliptic functions. In the following, we will show these special solutions; for further details on the Painlev\'e--Gambier classification, the reader may refer to, for example, \cite{kudryashov2017connections,ince1956ordinary}.

Under condition \eqref{eq:CC1}, the ODE \eqref{eq:Heq} becomes 
\begin{equation}\label{eq:Heq1}
\frac{d^2 H}{d\theta^2}-C_{0}\frac{d H}{d\theta} + H^{2} + \frac{6C_0^2}{25} \,H=0, \quad \text{i.e.,}\quad  \wt C_1= -2+ \frac{3C_0^2}{25},
\end{equation}
and
\begin{equation}\label{eq:Heq2}
\frac{d^2 H}{d\theta^2}-C_{0}\frac{d H}{d\theta} + H^{2} - \frac{6C_0^2}{25} \,H=0, \quad \text{i.e.,}\quad  \wt C_1= -2- \frac{3C_0^2}{25}.
\end{equation}
One can readily verify that the two equations are related by the transformation $H\mapsto H-6C_0^2/25$ (from \eqref{eq:Heq1} to \eqref{eq:Heq2}). Substituting solutions of either equation into \eqref{eq:hhh1}  yields the same solution $h(z)=h_L(z)+h_r(z)$ where $h_r(z)=H(\theta)$, and consequently leads to the same solution of the SNS equations \eqref{eq:NS_0}.

From now on, we will only study the ODE \eqref{eq:Heq1}, which has the first integral
\begin{equation}\label{eq:1stint}
   \left(\frac{d H}{d\theta} -\frac{2C_0}{5}H\right)^2 +\frac{2}{3}H^3=C_2 \exp\left(\frac{6C_0}{5}\theta\right), 
\end{equation}
where $C_2$ is an arbitrary constant. 
Consider the following transformations
\begin{equation}\label{eq:HP}
    H(\theta)= -\frac{6C_0^2}{25} \tau^2 \,\mathfrak{P}(\tau), 
    \quad \tau = \exp\left( \frac{C_0}{5}\theta \right),
\end{equation}
and the first integral \eqref{eq:1stint} becomes the following ODE:
\begin{equation}\label{eq:ell}
    \left(\frac{d\,\mathfrak{P}(\tau)}{d \tau}\right)^2 = 4\,\mathfrak{P}^3(\tau) - g_3
\end{equation}
with $g_3$ an arbitrary real constant. 

If $d\,\mathfrak{P}(\tau)/d\tau \equiv 0$, then\eqref{eq:ell} admits the constant solution
$\mathfrak{P}(\tau)\equiv\sqrt[3]{g_3/4}$, and \eqref{eq:HP} yields
\begin{equation*}
    H(\theta) =\begin{cases}
        0, & \text{if}\,\,\, g_3=0,\\
       -\frac{6C_0^2}{25}\sqrt[3]{\frac{g_3}{4}}
\exp\left(\frac{2C_0}{5}\theta\right),
& \text{if}\,\,\, g_3\not =0,
    \end{cases}
\end{equation*}
which, however, is \emph{not} a nonconstant solution of \eqref{eq:Heq1}.
Thus $d\,\mathfrak{P}(\tau)/d\tau \not\equiv 0$, and the nonconstant smooth real solution of \eqref{eq:ell}
can be written as
\[
\mathfrak{P}(\tau)=\wp(\tau+C;0,g_3)
\]
for any $C\in\mathbb R,$ where $\wp(\tau+C;0,g_3)$ stands for the Weierstrass elliptic functions for $g_3 \not=0$ (see \cite[Chap.18]{AbramowitzStegun} or
\cite[Equations 23.3.10 and 23.3.12]{ReinhardtWalkerDLMF}), and $\wp(\tau+C;0,0)=(\tau+C)^{-2}$ for $g_3=0$. 
Furthermore, if $g_3=C=0$,  \eqref{eq:HP}  implies that $H(\theta)\equiv-6C_0^2/25$, which contradicts the nonconstant assumption of $h_r.$  
Consequently, the nonconstant smooth solution of
\eqref{eq:Heq1} under consideration is given by
\begin{equation}\label{eq:HHH}
    H(\theta)
    =-\frac{6C_0^2}{25}
    \exp\left(\frac{2C_0}{5}\theta\right)
    \wp\left(
        \exp\left(\frac{C_0}{5}\theta\right)+C;
        0,g_3
    \right)
\end{equation}
for $(C,g_3)\in\Rs=\R^2\backslash \{(0,0)\}$.

Finally, we obtain the second type non-radial self-similar solution of the SNS equations \eqref{eq:NS_0} from \eqref{eq:hhh1} and \eqref{eq:HHH} as follows
\begin{equation}\label{eq:uuuvvv}
    \begin{aligned}
            u(x,y)&=u_L(x,y) -\frac{6C_0^2x}{25(x^2+y^2)} \exp\left(\frac{2C_0}{5}\arctan \frac{y}{x} \right) \cdot \wp \left( \exp\left(\frac{C_0}{5}\arctan \frac{y}{x} \right) +C;0,g_3 \right),\\
        v(x,y)&=v_L(x,y) -\frac{6C_0^2y}{25(x^2+y^2)} \exp\left(\frac{2C_0}{5}\arctan \frac{y}{x} \right) \cdot  \wp \left( \exp\left(\frac{C_0}{5}\arctan\frac{y}{x} \right) +C;0,g_3 \right),
    \end{aligned}
\end{equation}
where $u_{L}$ and $v_{L}$ are given by \eqref{eq:U_L1} with $\wt C_1=-2+ 3C_0^2/25$. 
In particular, for $g_3=0$ and $C\neq0$, the solution formula \eqref{eq:uuuvvv} reduces to
\begin{equation}\label{eq:uv-special21-exp}
    \begin{aligned}
        u(x,y) &= u_{L}(x,y)-\frac{6C_{0}^{2}x}{25(x^2+y^2)} \frac{1}{\Big( 1+C\exp\big(-\frac{C_0}{5}\arctan\,(y/x) \big)\Big)^2}, \\
        v(x,y) &= v_{L}(x,y)-\frac{6C_{0}^{2}y}{25(x^2+y^2)} \frac{1}{\Big( 1+C\exp\big(-\frac{C_0}{5}\arctan\,(y/x) \big)\Big)^2}.
    \end{aligned}
\end{equation}

\begin{remark} 
 In the above, we obtained the solution structure of the reduced ODE \eqref{eq:h_C0} when a special linear solution exists under the condition \eqref{eq:con}, and all its solutions can be associated with the Li\'enard equation \eqref{eq:Heq}. However, if $C_0^2 >4-2C_1$ and if \eqref{eq:h_C0} admits a nonlinear special solution $h_s(z)$, then $h_r(z)=h(z)-h_s(z)$ satisfies 
   \begin{equation}\label{eq:hr_hs}
(z^{2}+1)^{2} \frac{d^2h_r}{dz^2} +(2z-C_{0})(z^{2}+1) \frac{d h_r}{dz }+h_{r}^{2} +(2h_s + 2 C_0 z+4 )h_{r}=0.
\end{equation}
Similarly, by defining $H(\theta):=h_{r}(\tan \theta)$ 
with $\theta:=\arctan\,z$, \eqref{eq:hr_hs} becomes a nonautonomous Li\'enard equation:
\begin{equation*}
\frac{d^2 H}{d\theta^2}-C_{0}\frac{d H}{d\theta} + H^{2} 
+ \big(2h_s(\tan \theta)+ 2 C_0 \tan \theta +4 \big) \,H=0.
\end{equation*}
Other special solutions may be obtained, but they are beyond the scope of the present paper.
\end{remark}

\section{Conclusion}

In this paper, we studied self-similar solutions of the planar SNS equations \eqref{eq:NS_0} without external forces. Motivated by the scaling invariance \eqref{eq:scasym}, we focused on homogeneous solutions of order $-1$, and worked on the punctured plane $\R^2_*=\R^2\setminus{(0,0)}$ to avoid the strong singularity of the convection term $\dv(\bU\otimes\bU)$ at the origin (cf. \cite[p. 7]{sverak2011landau}).

Our main results provide a complete description of smooth \emph{radial} self-similar profiles with respect to the scaling symmetry in general conical domains $\Omega\subset\R^2_*$, without prescribing boundary conditions on $\partial\Omega$. Under the radial constraint $xv=yu$ \eqref{cdt:xnot0}, \eqref{eq:NS_0} reduces to an integrable one-dimensional ODE, which can be solved in closed form. In particular, Theorem \ref{thm:main-1} classifies all local smooth radial self-similar solutions in “small” cones (opening angle $\alpha(\Omega)\leq \pi$): every such solution is explicitly represented by a profile function $\kappa$ listed in Table \ref{tab:kap}. Theorem \ref{thm:main-2} further shows that the corresponding profiles in “large” cones (opening angle $\alpha(\Omega)\in (\pi,2\pi]$) are obtained by real-analytic continuation of those in Theorem \ref{thm:main-1}. 
Moreover, whenever these local solutions can be smoothly extended to the whole
plane, they recover precisely the classical radial Jeffery--Hamel solutions.
This provides a complementary point of view to the classical Jeffery--Hamel
literature and, in particular, supplements the results in
\cite{bang2024self,fraenkel1962laminar,sverak2011landau,guillod2015generalized}.
 As a further application, our classification  shows that the uniqueness result in \cite{kobayashi2014jefferyhamel} may fail in the absence of the no-slip condition.

\medskip

Beyond the radial regime, the structure of genuinely \emph{non-radial} self-similar solutions is substantially more delicate. When $xv=yu+C_0$ with $C_0\neq0$ (i.e., \eqref{cdt:C_0}), we showed that a particular class of non-radial self-similar solutions can be represented in terms of an angular profile $H(\theta)$ satisfying the Li\'enard equation \eqref{eq:H_pm} (see Theorem \ref{thm:abel-reduction}). Although this equation is not integrable in general, we constructed a special solution \eqref{eq:HHH} expressed in terms of Weierstrass elliptic functions, which in turn yields the solution \eqref{eq:uuuvvv} of the SNS equations.

Several directions remain open. A natural extension is to incorporate homogeneous external forces into the stationary framework, which may generate new self-similar classes; some preliminary reductions in this direction are discussed in Section \ref{sec:symm}. Another important direction is to address the non-radial configurations not covered by our explicit constructions. Especially, when \eqref{eq:h_C0} has no linear solutions, it is unclear whether one can still obtain exact self-similar solutions to the SNS equations.

\section*{Acknowledgments}
LP is partially supported by JSPS
KAKENHI (JP24K06852), JST CREST (JPMJCR24Q5), and Keio University (Fukuzawa Fund). 
PZ is partially  supported by National Key R$\&$D Program of China under grant 2021YFA1000800 and by National Natural Science Foundation of China under Grants  No. 12421001, No. 12494542 and No. 12288201.
XZ was supported by the Fundamental Research Funds for the Central Universities. 

LP and XZ gratefully acknowledge the support and hospitality of Tongji University and Keio University, where a significant portion of this manuscript was prepared.

\section*{Data Availability}
Data sharing not applicable to this article as no datasets were generated or analysed during
the current study.

\section*{Declarations}

{\bf Conflict of interest} \,\,On behalf of all authors, the corresponding author states that there is no conflict of
interest.



\begin{appendices}

\section{Symmetric radial self-similar solutions in wedge-shaped channels}
\label{app:example}

In this appendix, we focus on the radial self-similar solution of the SNS \eqref{eq:NS_0} in the following wedge-shaped channel
\[
\Omega_{\alpha}=\{(r,\theta):r>0,\ -\alpha<\theta<\alpha\}
\,\,\, \text{with}\,\,0<\alpha<\pi/4.
\]
Let us first recall the following partial result
\footnote{In the setting of the present study, the kinematic viscosity coefficient $\nu=1$.}
from \cite[Theorem 1.1]{kobayashi2014jefferyhamel} 
on the uniqueness of symmetric radial self-similar solutions of \eqref{eq:NS_0} in $\Omega_{\alpha}$. 
\begin{prop}[see Theorem 1.1 in \cite{kobayashi2014jefferyhamel}]\label{prop:Kobayashi}
For $\alpha \in (0,\pi/4),$  let us define 
\begin{equation*}
\begin{aligned}
    C(\alpha)&:=1-\cos 2\alpha + \sin 2\alpha \tan 2\alpha,\\
    K(\alpha)&:=  \frac{\alpha(1-\cos 2\alpha)}{|\tan 2\alpha-2\alpha||\cos 2\alpha|},\\
    \gamma_0(\alpha)&:= \frac{1}{C(\alpha)K(\alpha)(2\alpha K(\alpha)+1)}.
\end{aligned}
\end{equation*}
Then the system \eqref{eq:NS_0} in $\Omega_{\alpha}$ admits only one radial symmetric self-similar profile of velocity field 
\footnote{Here, $r=\sqrt{x^2+y^2}$ and $\be_{r}=(\cos \theta, \sin \theta)^{\top}$ as introduced in Section \ref{sec:1.1}.}
$\bU=f(\theta)\be_r/r$ vanishing as $r\to \infty$ and satisfying the following conditions:
\begin{equation}\label{cdt:f_gamma}
    \begin{aligned}
     & f \in C\big([-\alpha,\alpha]\big)\cap C^\infty\big((-\alpha,\alpha)\big),
     \quad f(-\theta)=f(\theta), \\
     & f(\pm\alpha)=0,\quad 
\int_{-\alpha}^{\alpha}f(\theta)\,d\theta=\gamma,\quad
\|f\|_{L^\infty((-\alpha,\alpha))}\le 2K(\alpha)|\gamma|
    \end{aligned},
\end{equation}
provided $0<|\gamma|<\gamma_0(\alpha)$.
\end{prop}

Now, we shall prove that the uniqueness criterion in Proposition \ref{prop:Kobayashi} may fail if we ignore the no-slip boundary conditions $f(\pm\alpha)=0$ in \eqref{cdt:f_gamma}.
To this end, we consider the branch of the singular solution (i.e., $\kappa \equiv C$) and the
branch for $(C_1,C_2)\in \Gamma_-$ listed in Table \ref{tab:kap}. 
In polar coordinates, we examine two distinct self-similar profiles of velocity fields $\bU_j=f_j(\theta)\be_r/r$ ($j=1,2$)  in $\Omega_{\alpha}$  with 
\[
\begin{aligned}
 f_1(\theta) & =\frac{\gamma}{2\alpha},   \\
 f_2(\theta) & =
-2+2\sqrt{2(2-C_1)}
-3\sqrt{2(2-C_1)}
\tanh^2\left(
\sqrt[4]{\frac{2-C_1}{2}}\,\theta
\right).
\end{aligned}
\]
Moreover, it is clear that 
$$f_j \in C\big([-\alpha,\alpha]\big)\cap C^\infty\big((-\alpha,\alpha)\big)
\quad\text{and}\quad  
f_j(\theta)=f_j(-\theta)$$ 
for $j=1,2$ and $\theta \in [-\alpha,\alpha].$ 
\smallbreak 

Next, we choose some parameter $C_1<2$ such that the flux condition in \eqref{cdt:f_gamma} holds, namely,
\begin{equation}\label{eq:cdt_gamma}
 \int_{-\alpha}^{\alpha} f_1(\theta)\,d\theta=\gamma=\int_{-\alpha}^{\alpha} f_2(\theta)\,d\theta.   
\end{equation}
To do this, in view of the equality
\[
\int_{-\alpha}^{\alpha}\tanh^2(a\theta)\,d\theta
=
\left[\theta-\frac{\tanh(a\theta)}{a}\right]_{-\alpha}^{\alpha}
=
2\alpha-\frac{2}{a}\tanh(a\alpha) \quad (a\not=0),
\]
the flux condition \eqref{eq:cdt_gamma} holds 
if we take the value of $\gamma$ as follows
\begin{equation}\label{eq:gamma_C1}
\begin{aligned}
\gamma
={}&-2\alpha\left(2+\sqrt{2(2-C_1)}\right)
+6\sqrt{2\sqrt{2(2-C_1)}}\,
\tanh\left(
\alpha\sqrt[4]{\frac{2-C_1}{2}}
\right).    
\end{aligned}    
\end{equation}
\smallbreak 

Finally, we shall see that the last condition in \eqref{cdt:f_gamma}, concerning the $L^{\infty}$ bound of $f_j$, may also hold for certain explicit values of $\alpha,$ $\gamma$ and $C_1.$   
For example, we take 
\[
\alpha=\frac{\pi}{6}
\quad \text{and} \quad 
C_1=\frac{129}{80}<2,
\]
which, together with \eqref{eq:gamma_C1}, imply that
\[
\begin{aligned}
 \gamma
&=
-\frac{\pi}{3}\left(2+\sqrt{\frac{31}{40}}\right)
+12\sqrt[4]{\frac{31}{160}}\,
\tanh\left(\frac{\pi}{6}\sqrt[4]{\frac{31}{160}}\right)
\approx -0.357,  \\
f_1(\theta)&=\frac{3\gamma}{\pi}
\quad \text{and}\quad 
f_2(\theta)
=
-2+2\sqrt{\frac{31}{40}}
-3\sqrt{\frac{31}{40}}\,
\tanh^2\left(\sqrt[4]{\frac{31}{160}}\,\theta\right)
\end{aligned}
\]
for $\theta \in [-\pi/6,\pi/6]$.
On the other hand, the constants introduced in
Proposition \ref{prop:Kobayashi} satisfy
\[
\begin{aligned}
C\left(\frac{\pi}{6}\right)
=2,\quad 
K\left(\frac{\pi}{6}\right)
=\frac{\pi}{2(3\sqrt3-\pi)},\quad 
\gamma_0\left(\frac{\pi}{6}\right)
=\frac{6(3\sqrt3-\pi)^2}
{\pi\bigl(6(3\sqrt3-\pi)+\pi^2\bigr)}
\approx 0.363.
\end{aligned}
\]
Consequently, the condition
$0<|\gamma|<\gamma_0(\pi/6)$ holds as well.

To verify the upper bound condition 
\begin{equation}\label{eq:upper}
    \max_{j=1,2}\|f_j\|_{L^\infty((-\alpha,\alpha))}
<2K\left(\frac{\pi}{6}\right)|\gamma|,
\end{equation}
it is straightforward to check that 
\[
\|f_1\|_{L^\infty((-\alpha,\alpha))}=\frac{3|\gamma|}{\pi}
\approx 0.341.
\]
As $f_2$ is negative and strictly decreasing on
$[0,\pi/6]$, we have 
\begin{equation*}
 \begin{aligned}
\|f_2\|_{L^\infty((-\alpha,\alpha))}
=-f_2\left(\frac{\pi}{6}\right)
=2-2\sqrt{\frac{31}{40}}
+3\sqrt{\frac{31}{40}}\,
\tanh^2\left(\frac{\pi}{6}\sqrt[4]{\frac{31}{160}}\right)
\approx 0.534.
\end{aligned}   
\end{equation*}
Since
\[
2K\left(\frac{\pi}{6}\right)|\gamma|
=\frac{\pi|\gamma|}{3\sqrt3-\pi}
\approx 0.545,
\]
we obtain \eqref{eq:upper}.

From the above discussion, $f_1$ and $f_2$ are two distinct profiles in $\Omega_{\pi/6}$ satisfying all the conditions in \eqref{cdt:f_gamma} except that $f_j(\pm \pi/6)=0$ for $j=1,2.$
Thus, the uniqueness criterion in Proposition \ref{prop:Kobayashi} may fail without assuming no-slip boundary conditions.

\section{Proof of Lemma \ref{lem:cos}}\label{prf lem:cos}
In this section, we prove Lemma \ref{lem:cos}. Firstly, we prove that $\omega_c \in C^{0}(\Rs)$.
For simplicity, let us denote 
\begin{equation*}
g_0(x,y):= 
           \mathrm{am}\left(
                   \sqrt{\dfrac{c-a}{6}} 
                   \wt \theta(x,y)+\sqrt{c-a}C , k_0 
                    \right), \quad k_0=\sqrt{\dfrac{c-b}{c-a}}
\end{equation*}
for any $(x,y)\in D:=\Rs\setminus R_0$. 
Notice that the auxiliary function $\wt\theta$ defined in \eqref{wt theta} belongs to $C^{\infty}(D)$ and then $g_0 \in C^{\infty}(D)$.
Hence, it suffices to verify that 
\begin{equation}\label{eq:omega_sc_1}
\lim_{y\to 0^{+}}  \omega_c (x,y)=   \lim_{y\to 0^{-}}  \omega_c (x,y)=\omega_c(x,0),\quad x>0.
\end{equation}
 To this end, we can easily see from the condition \eqref{cdt:K} that 
\begin{equation*}
    \begin{aligned}
        \lim_{y\to 0^+} g_0(x,y)
    &=\mathrm{am}\left(\sqrt{c-a}C , k_0 \right),\\
         \lim_{y\to 0^-} g_0(x,y)
         &=\mathrm{am}\left(\sqrt{c-a} \,C+2nK(k_0), k_0 \right),
    \end{aligned}
\end{equation*}
which imply
\begin{equation}\label{eq:omega_sc_1_1}
\begin{aligned}
\lim_{y\to 0^{+}} \omega_c(x,y) 
&=\cos \left( 2\, \mathrm{am}\left(\sqrt{c-a} \,C, k_0 \right)\right),\\
\lim_{y\to 0^{-}} \omega_c(x,y)
&=\cos \left( 2\, \mathrm{am}\left(\sqrt{c-a} \,C+2nK(k_0), k_0 \right)\right).
\end{aligned}
\end{equation}
Therefore, \eqref{eq:omega_sc_1} follows from \eqref{eq:omega_sc_1_1}.
\smallbreak 

Secondly, we prove the continuity of partial derivatives of $\omega_c$.
In view of the computation in \eqref{partial theta smooth}, we observe that 
$\pd_{x,y}^{\alpha} \wt \theta (x,y) \in C^{\infty}(\Rs)$ 
for any multi-index $\alpha\in \BN_0^2$ with $|\alpha| \geq 1$. 
Moreover, Bell's formula implies that 
\begin{equation*}
\begin{aligned}
   \pd_{x,y}^{\alpha} g_0 (x,y) &= \sum_{\ell=1}^{|\alpha|} 
   \frac{\pd^{\ell}}{\pd \mu^{\ell}} 
   \am (\mu,k_0) \Big|_{\substack{\mu =\sqrt{\left(c-a\right)/6} \,\wt \theta(x,y) +\sqrt{c-a}}C} \\  
   & \qquad \times \left( 
 \sum_{\substack{\alpha_1+\cdots+\alpha_{\ell}=\alpha \\|\alpha_1|,\dots,|\alpha_{\ell}|\geq 1}} 
    C_{\alpha_1,\dots,\alpha_{\ell}}^{\ell}
    \pd_{x,y}^{\alpha_1} \wt \theta(x,y)  \cdots \pd_{x,y}^{\alpha_{\ell}}  \wt \theta(x,y)  \right),
\end{aligned}
\end{equation*}
where the constant $C_{\alpha_1,\dots,\alpha_{\ell}}^{\ell}$ depends on $\alpha_1,\dots,\alpha_{\ell}$.
Moreover, for fixed $k_0$, $\operatorname{dn}\,(\mu,k_0)=\pd \operatorname{am}\,(\mu,k_0)/\pd \mu$ is a periodic function with
a period $2K(k_0)$, and so are the higher-order derivatives of $\operatorname{dn}$. 
Therefore, analogously to the proof of \eqref{eq:omega_sc_1}, we obtain from \eqref{cdt:K} that  
\begin{equation*}
    \begin{aligned}
       &\lim_{y\to 0^{+}} \left( \frac{\pd^{\ell'}}{\pd u^{\ell'}} \dn  \right) (\mu,k_0)\Big|_{\substack{\mu = \sqrt{\left(c-a\right)/6} \, \wt\theta(x,y)+\sqrt{c-a}}C}\\
=&\left( \frac{\pd^{\ell'}}{\pd u^{\ell'}} \dn  \right)\left(\sqrt{c-a} \,C, k_0 \right)  \\
=&\lim_{y\to 0^{-}} \left( \frac{\pd^{\ell'}}{\pd \mu^{\ell'}} \dn  \right) (\mu,k_0)\Big|_{\substack{\mu = \sqrt{\left(c-a\right)/6}\,\wt\theta(x,y)+\sqrt{c-a}}C},
    \end{aligned}
\end{equation*}
which yields $g_0\in C^{\infty} (\Rs)$.
\smallbreak 

Finally, it is straightforward to show that 
\begin{equation*}
\begin{aligned}
     \pa_{x,y}^{\alpha} \omega_c (x,y)
     &= \sum_{|\ell|=1}^{|\alpha|} 2^{\ell}\cos \left(2g_0(x,y)+\frac{\ell\pi}{2}\right) \\
     & \qquad \times\left( 
     \sum_{\substack{\alpha_1+\cdots+\alpha_{\ell}=\alpha\\
     |\alpha_1|,\dots,|\alpha_{\ell}|\geq 1}} 
    C_{\alpha_1,\dots,\alpha_{\ell}}^{\ell}
    \pd_{x,y}^{\alpha_1} g_0(x,y) \cdots  \pd_{x,y}^{\alpha_{\ell}}  g_0(x,y)  \right) \in C^{0}(\Rs),
\end{aligned}
\end{equation*}
which implies $\omega_c\in C^{\infty}(\Rs)$.
This completes the proof of Lemma \ref{lem:cos}.

\end{appendices}


\addcontentsline{toc}{section}{References}

\end{document}